\documentclass[12 pt]{amsart}
\usepackage{amsfonts}
\usepackage{mathrsfs}
\usepackage{tikz}
\usepackage{pinlabel}   
\usepackage{graphicx}
\usepackage{verbatim}
\usepackage[abs]{overpic}
\usepackage{amsmath,amssymb}
\usepackage{geometry}
\newtheorem{theorem}{Theorem}[section]
\newtheorem{lemma}[theorem]{Lemma}
\newtheorem{corollary}[theorem]{Corollary}
\newtheorem{proposition}[theorem]{Proposition}

\theoremstyle{definition}
\newtheorem{definition}[theorem]{Definition}

\theoremstyle{remark}
\newtheorem{remark}[theorem]{Remark}
\numberwithin{equation}{section}

\begin{document}
\title{Stein fillings of planar open books}
\author{Amey Kaloti }
\address{School of Mathematics \\ Georgia Institute of Technology\\Atlanta, GA 30332, USA}
\email{ameyk@math.gatech.edu}

\bibliographystyle{plain}

\begin{abstract}
We prove that if a contact manifold $(M,\xi)$ is supported by a planar open book, then Euler characteristic and signature of any Stein filling of $(M,\xi)$ is bounded. We also prove a similar finiteness result for contact manifolds supported by spinal open books with planar pages. Moving beyond the geography of Stein fillings, we classify fillings of some lens spaces.
\end{abstract}

\maketitle

\section{Introduction}

\label{introduction}
Understanding filling properties of contact manifolds has been a very active area of research in contact topology. Part of the interest in fillings comes from the fact that Stein fillings and strong fillings can be used in cut-and-paste constructions in symplectic geometry. In addition, certain types of fillings are related to interesting properties. For example, a weakly fillable contact structure is tight and a strongly fillable one is known to have non-zero Heegaard Floer contact invariant. In this paper we are primarily concerned with the geography and the classification of Stein fillings of contact structures supported by planar open books.

\subsection{Euler Characteristic and Signature of Stein fillings}

Recall that $X$ is called a Stein filling of a contact $3$ manifold, $(M, \xi)$, if $X$ is a sub-level set of a plurisubharmonic function on a Stein surface and the contact structure induced on the $\partial X$ by the complex tangencies is contactomorphic to $(M,\xi)$. If we denote the Euler characteristic and signature of $X$ by $\chi(X)$ and $ \sigma(X)$, respectively, then the geography problem is to determine the following set
\[\mathcal{C}_{(M,\xi)} = \{(\sigma(X), \chi(X)) | X \, \text{is a Stein filling of}\,(M,\xi)\}.\]

In particular, it is interesting to know whether this set is finite. It has been a conjecture of Stipsicz~\cite{Stipsicz_Geoegraphy_of_stein_fillings}, that $\mathcal{C}_{(M,\xi)}$ is finite for any $(M,\xi)$. Stipsicz proved the conjecture for manifolds $(M,\xi)$, with the property that every Stein filling has a vanishing $b_2^+$. In addition, he also verified the conjecture for any contact structure on circle bundles over surface of genus $g$ and Euler number $n$, such that $|n| > 2g -2$. 

We have the following theorem when $(M,\xi)$ is supported by planar open book decomposition.

\begin{theorem}
\label{Finite_euler_characteristic}
Let $(M,\xi)$ be a tight contact manifold supported by planar open book. Then the set $\mathcal{C}_{(M,\xi)}$ is finite. In particular, there exists a positive integer $N$ such that signature and Euler characteristic of $X$ satisfy, $|\sigma(X)| < N$ and $|\chi (X)| < N$ for any Stein filling $(X,J)$ of $(M,\xi)$.

\end{theorem}

Note that if $\chi(X)$ is bounded from above, then due to a result of Stipisicz~\cite{Stipsicz_Geoegraphy_of_stein_fillings}, $\sigma(X)$ is bounded from below. Above theorem gives an upper bound as well.

\begin{remark}
After the paper was submitted, Jeremy Van Horn-Morris pointed out to the author that Plamenevskaya had proved finiteness of Euler characteristic in her paper~\cite{Plamenevskaya_Legendrian_surgeries}. The author was unaware of this result.
\end{remark}

It is known that all contact structures on lens spaces $L(p,q)$ are supported by planar open books. Also, if the Euler number $e_0$, satisfies $e_0 \neq -1,-2$, then all tight contact structures on small Seifert fibered spaces are supported by planar open books~\cite{Schoenenberger_thesis_upenn}. Hence we obtain the following immediate corollary of above theorem.

\begin{corollary}
Let $\xi$ be a tight contact structure on $L(p,q)$ or on a small Seifert fibered space $M(e_0;r_1,r_2,r_3)$ with $e_0 \neq -1,-2$. There exists a positive integer $N$ such that for any Stein filling $(X,J)$ of $\xi$, $|\chi(X)| < N$ and $|\sigma(X)| < N$.

\end{corollary}

\begin{remark}
The conjecture in full generality has been recently proven to be false by Baykur and Van Horn-Morris~\cite{Baykur_VHMorris_infinitely_many_euler_characteristic}. The Stein $4$-manifolds Baykur and Van Horn-Morris construct to disprove the conjecture are given by Lefschetz fibrations with non-closed fiber of genus $g \geq 2$ and non-closed base of genus $h \geq 1$. In this case, one has to talk about more general open books called spinal open books. In addition, Baykur and Van Horn-Morris in~\cite{Baykur_VHMorris_Arbitrarily_Long_Factorizations} factorize a positive Dehn multitwist along boundary components in arbitrarily large number of positive Dehn twists along non-separating curves. The surface they exhibit these factorizations is any surface with genus $g \geq 11$ andd two boundary components. This given another counterexample to the conjecture of Stipsicz. The result in~\cite{Baykur_VHMorris_Arbitrarily_Long_Factorizations} was improved to any surface with genus $g \geq 3$ and two boundary components in~\cite{Dalyan_Korkmaz_Pammuk_Infinitely_many_factorizations}.
\end{remark}

Note that the Theorem~\ref{Finite_euler_characteristic} only guarantees an upper bound on Euler characteristic. It is an interesting question to see if one can realize an exact upper bound on the Euler characteristic.  We use the methods developed in this paper to get an explicit upper bound on the Euler characteristic of a particular contact structure. Let $C = C_1 \cup C_2 \cup \dots \cup C_n$ denote a configuration of symplectic spheres in a symplectic manifold $(X,\omega)$ intersecting $\omega$-orthogonally according to a connected plumbing graph $\Gamma$ with negative definite intersection form $Q = (q_{ij})= [C_i] \cdot [C_j]$. We assume that there are no edges connecting a  vertex to itself. Suppose that for each row in $Q$, we have a non positive sum $\sum_j q_{ij} \leq 0$. It follows from a result of Gay and Mark~\cite{Gay_Mark_Convex_Configuration}, that any neighbourhood of such a configuration of symplectic spheres $C$ contains a neighbourhood $(Z,\eta)$ of $C$ with strong convex boundary. The boundary $M$ of $(Z,\eta)$ has a natural contact structure which we denote by $\xi_{pl}$. 

\begin{theorem}
\label{Thm:Sphere_Plumbing}
Let $(M,\xi_{pl})$ be contactomorphic to the boundary of $(Z,\eta)$ which is a plumbing of spheres as defined above. If $(X,J)$ is a strong symplectic filling of $(M,\xi_{pl})$, then $\chi(X) \leq \chi(Z)$. 

\end{theorem}

This theorem answers a special case of a question raised by Starkston. See Question~$6.2$ in~\cite{Starkston_Symplectic_Fillings}.

Our approach to studying fillings of contact structures supported by planar open books is motivated by the following result of Giroux~\cite{Giroux_Correspondence}, Loi-Piergallini~\cite{Loi_Piergallini_Lefschetz_fibrations}, Akbulut-Ozbagci~\cite{Akbulut_Ozbagci_Stein_Surface_Lefschetz_fibrations} that a contact manifold $(M,\xi)$ is Stein fillable if and only if there exists a compatible open book decomposition $(\Sigma, \Phi)$ such that the monodromy is written as a product of positive Dehn twists. Given such a factorization of $\Phi$ in terms of positve Dehn twists, one can construct Stein filling as an allowable Lefschetz fibration of disk $D$ with fiber $\Sigma$. Given a Stein filling of $(M,\xi)$, one can construct Lefschetz fibration $X$ (see~\cite{Akbulut_Ozbagci_Stein_Surface_Lefschetz_fibrations, Loi_Piergallini_Lefschetz_fibrations}) such that $\partial X = M$ with monodromy written as a positive factorization corresponding to vanishing cycles of the Lefschetz fibration. So to classify the Stein fillings of $(M, \xi)$, one will have to find all compatible open books and then find all possible ways of factorizing a given monodromy in terms of positive Dehn twists. Note that there are examples of Stein fillable manifolds such that a compatible open book does not admit any positive factorizations, see~\cite{Baker_Etnyre_VHMorris_Rational_open_books, Wand_Positive_factorizations}. But in the case of manifolds supported by planar open books this problem is approachable due to the following theorem of Wendl.

\begin{theorem}[Wendl 2010, \cite{Wendl_Planar_open_books}]
Suppose $(M,\xi)$ is supported by a planar open book. Then every strong symplectic filling $(X,\omega)$ of $(M,\xi)$ is symplectic deformation equivalent to a blow-up of an allowable Lefschetz fibration compatible with the given open book of $(M,\xi)$.

\end{theorem}

In other words, Wendl's theorem says that to classify the Stein fillings of planar open books one has to find all possible ways of factoring the given monodromy in terms of positive factors. This strategy was used by Plameneveskaya and Van Horn-Morris~\cite{Plamenevskaya_VHMorris_Planar_open_books} to classify the Stein fillings of lens spaces $L(p,1)$ with virtually tight contact structures. The authors prove the result by first considering the problem in the abelianization of mapping class group. This gives a restriction on any positive factorization of the monodromy. After this they complete the classification by using surgery characterization of unknot~\cite{Kronheimer_Mrowka_Oszvath_Szabo} and the Legendrian simplicity of unknot~\cite{Eliashberg_Fraser_Unknot_Classification}. We give another proof of this result by using mapping class group techniques to conclude the number of distinct positive factorizations of a given monodromy. We also use these techniques to classify fillings of other contact structures and bound the geography of contact structures supported by planar open books.

Recently, Wendl's result was extended to a wider class of contact manifolds in~\cite{Lisi_VHMorris_Wendl_Spinal_Open_Books}. These contact structures are supported by spinal open books with planar pages. We have a theorem which is a generalization of Theorem \ref{Finite_euler_characteristic} to this case. See Section \ref{spinal_open_books_basics} for basics of spinal open books.
\begin{theorem}
\label{spinal_finiteness_of_euler_characteristic}
Let $(M,\xi)$ be a contact structure supported by spinal open book with connected planar pages. Then $\mathcal{C}_{(M,\xi)}$ is finite. In particular, there exists a positive integer $N$ such that for any Stein filling $(X,J)$ of $(M,\xi)$, $|\chi(X)|, |\sigma(X)| < N$.
\end{theorem}

\subsection{Stein fillings of lens spaces}

We now focus on the classification problem for the Stein fillings. The classification problem has been solved for $(S^3,\xi_{std})$, lens spaces $(L(p,q),\xi_{std})$ \cite{Lisca_Classification_of_Stein_fillings_on_lens_spaces}, \cite{McDuff_Rational_ruled_surfaces} and $(L(p,1),\xi_{vot})$~\cite{Plamenevskaya_VHMorris_Planar_open_books}. Here $\xi_{std}$ denotes a universally tight contact structure and $\xi_{vot}$ denotes a virtually overtwisted contact structure. The classification has also been achieved for links of singularities by work of Ohta and Ono~\cite{Ohta_Ono_Classifications_of_Stein_fillings_1, Ohta_Ono_Classifications_of_Stein_fillings_2}. By extending techniques developed by Plamenevskaya and Van Horn-Morris we classify the Stein fillings of lens spaces $L(p(m+1)+1,(m+1))$ with any tight contact structure. 

Recall that a lens space admits a unique universally tight contact structure upto contactomorphism and at most $2$ upto isotopy~\cite{Giroux_lens_space_classification, Honda_Classification_1}. One of our goals is to provide the classification of the Stein fillings of some virtually overtwisted lens spaces. We start out by proving a weaker version of the following theorem of Plamenevskaya and Van Horn-Morris.

\begin{theorem}
Every virtually overtwisted contact structure on $L(p,1)$ has a unique Stein filling up to symplectomorphism.
\end{theorem}

\begin{remark}
The original theorem in~\cite{Plamenevskaya_VHMorris_Planar_open_books} proves a stronger version of this result. The authors there prove the above classification up to symplectic deformation.

\end{remark}
Our proof of this result uses Wendl's result and mapping class group techniques. One advantage of this approach is that we can extend this result to a wider class of virtually overtwisted lens spaces.

\begin{theorem}
\label{virtually_overtwisted_lens_spaes}
Let $\xi$ be a contact structure on lens space $L(p(m+1)+1,(m+1))$. If $\xi$ is:
\begin{enumerate}
\item Virtually overtwisted, then $\xi$ has a unique Stein filling upto symplectomorphism.
\item Universally tight and $p \neq 4,5,\dots,(m+4)$, then $\xi$ has a unique Stein filling upto symplectomorphism.
\item Universally tight and $p = 4,5,\dots,(m+4)$, then $\xi$ has at least two Stein fillings upto symplectomorphism.
\end{enumerate}
 
\end{theorem}
  
The Legendrian surgery diagrams for these manifolds are given in Figure~\ref{fig:legendrian_surgery_L(2p-1,2)}

\begin{figure}[ht!]

%\labellist
%\small\hair 2pt
%\pinlabel 
%\pinlabel $m$ at 6 308
%\pinlabel $(-1)$ at 190 174
%\pinlabel $(-1)$ at 198 236
%\pinlabel $(-1)$ at 200 277
%\pinlabel $\mathcal{K}$ at 407 252
%\pinlabel $(-1)$ at 354 782
%\pinlabel $(-1)$ at 354 782
%\pinlabel $r$ at 578 531
%\endlabellist

\centering
\begin{overpic}[scale=0.40,tics=10]
{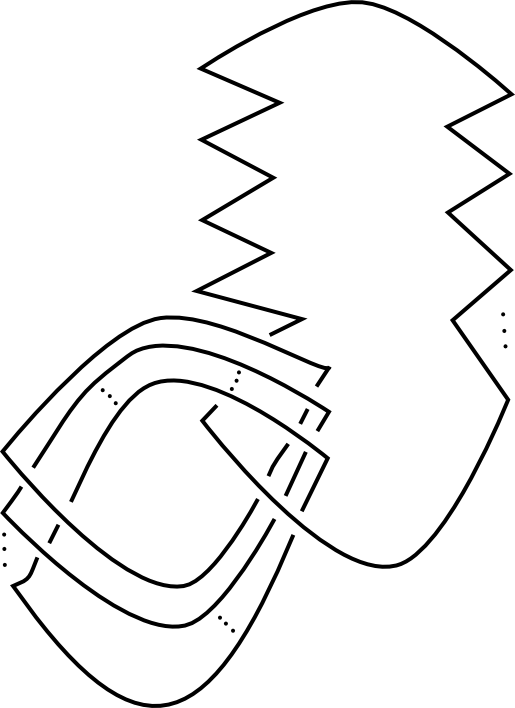}
\put(0, 3){ $(-1)$}
\put(40, 12){$(-1)$}
\put(20, 120){$(-1)$}
\put(-12, 42){$m$}
\put(130, 37){$\mathcal{K}$}
\put(165, 120){$r$}
\put(110, 210){$(-1)$}

\end{overpic}
%\includegraphics[width = 8cm]{legendrian_surgery_L(2p-1,2)}
%\caption{Surgery diagram for lens spaces $L(p(m+1)+1, m+1)$. In the diagram $\mathcal{K}$ is the maximal $tb$ unknot stabilized positively $r$ times and negatively $p-r-1$ times. As we vary $r$ from $0$ to $p-1$ we get all the contact structures on these lens spaces.}

\caption{Surgery diagram for lens spaces $L(p(m+1)+1, m+1)$. In the diagram $\mathcal{K}$ is the maximal $tb$ unknot stabilized positively $r$ times and negatively $p-r-1$ times. As we vary $r$ from $0$ to $p-1$ we get all the contact structures on these lens spaces.}

\label{fig:legendrian_surgery_L(2p-1,2)}
\end{figure}
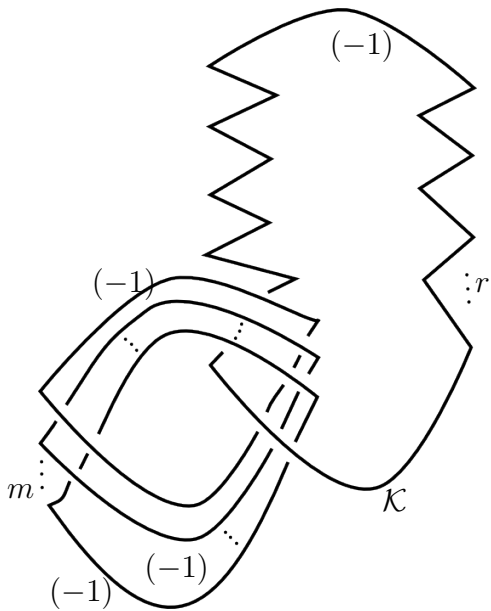

\subsection{Limitations of the technique}
We have used Wendl's result to study the geography and the classification problem for the Stein fillings. To use this technique effectively we need tools from mapping class groups. In practice, it is very hard to find all the explicit factorizations of a given word completely. On the other hand, knowing all the  factorizations of a given word need not give any idea about the number of the Stein fillings. This is due to the fact using Wendl's result one can get an upper bound on the number of the Stein fillings. We give an example of a contact structure such that the monodromy has arbitrarily large number of positive factorizations and yet having only finitely many Stein fillings. Let $Map(\mathbf{D_4}, \partial \mathbf{D_4})$  be mapping class group of a sphere with $5$ open disks removed.

\begin{theorem}
\label{curious_example}
There exists infinitely many positive factorizations of a word $\Phi$ written in terms of positive Dehn twists in $Map(\mathbf{D_4}, \partial \mathbf{D_4})$. However, the contact structure associated to the open book decomposition $(\mathbf{D}_4,\Phi)$, has at most finitely many Stein fillings.
\end{theorem}

Rest of this paper is organized as follows. In Section~\ref{background} we review basics of contact geometry and mapping class groups as required in this paper. In Section~\ref{lantern_relations} we prove technical lemmas which characterize lantern relation in a planar surface with $4$ boundary components. As the reader will see, these characterizations help us in proving all the classification results stated in above. In Section~\ref{combinatorial_arguments} we finish the proof of main results about the classification of the Stein fillings of lens spaces by using  combinatorial arguments. In Section~\ref{euler_characteristic_signature} we give the proofs of Theorem~\ref{Finite_euler_characteristic} and Theorem~\ref{spinal_finiteness_of_euler_characteristic}. In Section~\ref{Section:Sphere_Plumbings}, we prove Theorem~\ref{Thm:Sphere_Plumbing}. In the final section we give the example stated in Theorem~\ref{curious_example}.

\textbf{Acknowledgements:} The author would like to thank his advisor John Etnyre for all the support and guidance over years and innumerable discussions regarding this project. The author would also like to thank Dan Margalit for useful discussions. In particular, for pointing to his paper~\cite{Margalit_Lantern_lemma} and for the suggestion that results there could be used to prove the uniqueness of the Stein filligs. Thank you to Tom Mark and Laura Starkston for asking question whose answer led to Theorem~\ref{Thm:Sphere_Plumbing}.  Finally author would like to thank Jeremy Van Horn-Morris for useful discussions regarding this project. The author was partially supported by the NSF grant DMS-0804820.

\section{Background}
\label{background}
In this section we review some basic notions in contact geometry and mapping class groups. For more on the contact geometry part see~\cite{Etnyre_Introductory_Lectures, Geiges_Introduction_to_contact_geometry}. For the basics of mapping class groups see~\cite{Farb_Margalit_Primer}. In this paper we will denote a surface of genus $g$ with $r$ boundary components  by $\Sigma_{g,r}$. We will denote a closed surface of genus $g$ by $\Sigma_g$. We will use $\mathbf{D}_n$ to denote a sphere with $n+1$ open disks removed. We number the boundary components as $b_1,b_2,\dots,b_{n+1}$. By fixing an outer boundary component, denoted by $b_{n+1}$, we can embed the sphere in $\mathbb{R}^2$. A mapping class group of a surface $\Sigma_{g,r}$ is a group of diffeomorphism upto isotopy such that the diffeomorphism fixes the boundary pointwise. We will denote mapping class group of a closed surface $\Sigma_{g}$ by $Map(\Sigma_{g})$ and that of a surface $\Sigma_{g,r}$ with $r$ boundary components by $Map(\Sigma_{g,r},\partial \Sigma_{g,r})$. 

\subsection{Open books and Legendrian surgery diagrams}
\label{background_surgery_open_book}
Recall that $(S^3, \xi_{std})$ is the unique Stein fillable contact structure on $S^3$. It is well known that any manifold can be obtained as $(\pm)1$-contact surgery along a Legendrian link in $(S^3,\xi_{std})$~\cite{Ding_Geiges_Legendrian_Surgery_Presentation_Contact_Structures}. There is a natural way of constructing an open book decompositions of manifolds using surgery descriptions. We briefly explain the procedure. We will denote contact manifold by $(M,\xi)$ and open book supporting it by $(\Sigma,\Phi)$ in this section.

For any Legendrian knot, $\mathcal{L} \subset (M,\xi)$, there is an open book decomposition supporting $\xi$ such that $\mathcal{L}$ sits on the page of this open book decomposition and the framing given by the page and $\xi$ agree, see~\cite{Etnyre_Lectures_on_open_boon_decompositions}. There is a natural way of defining an open book decomposition for a manifold $(M',\xi')$ obtained by performing $(\pm)1$-contact surgery on a Legendrian knot in $(M,\xi)$. This is provided by the following theorem, see~\cite{Etnyre_Lectures_on_open_boon_decompositions}.

\begin{theorem}
Suppose $(M,\xi)$ is a contact manifold supported by $(\Sigma, \Phi)$. Let $L$ be a Legendrian knot contained in the page $\Sigma$. Then the contact manifold obtained by performing a $(\pm)1$-contact surgery is supported by an open book decomposition given by $(\Sigma,\Phi \circ \tau_L^{\mp})$. 
\end{theorem}

Given a Legendrian knot $L$, there is a natural operation called positive/negative stabilization of $L$ that can be used to get another Legendrian knot in the same knot type. For a Legendrian knot in $\mathbb{R}^3$ with its standard contact structure, positive(negative) stabilization is achieved by adding a zigzag to the front projection of the Legendrian knot such that rotation number of the Legendrian knot increases(decreases)  by $1$. Stabilization is a well-defined operation, that is, it does not depend on the point at which zigzags are added.  One can define this operation on in terms of open books as follows. See~\cite{Etnyre_Lectures_on_open_boon_decompositions} for details and proof of the lemma.

\begin{lemma}
Let $(\Sigma, \Phi)$ be an open book decomposition supporting the contact structure $\xi$ on $M$. Suppose $L$ is a Legendrian knot in $M$ that lies in the page $\Sigma$. If we stabilize $(\Sigma,\Phi)$ as shown in Figure~\ref{Knot_stabilization}, then we may isotop the page of the open book so that positive(negative) stabilization appear on the page $\Sigma$ as shown in Figure~\ref{Knot_stabilization}.
\end{lemma}

\begin{figure}[ht]
\begin{center}
\begin{overpic}%[grid,tics=10]
{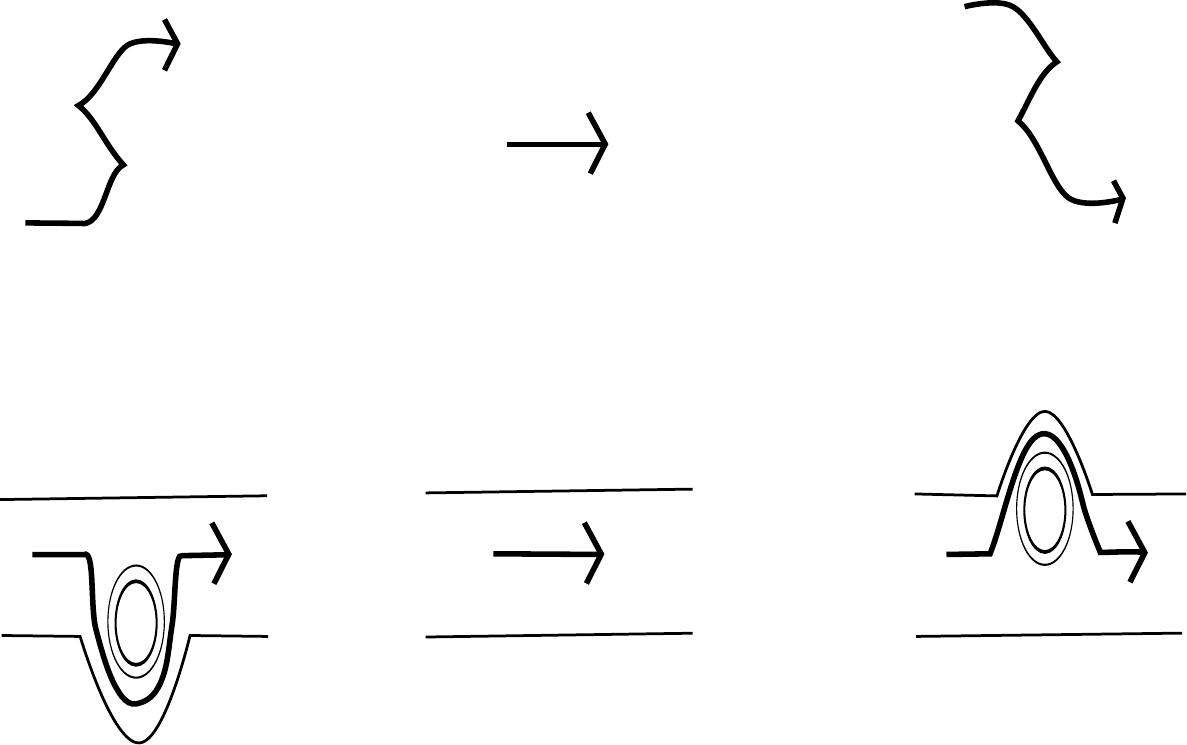}

\end{overpic}	

\caption{Stabilizing the knot on the page of an open book. Middle figure shows a Legendrian knot on a page of open book. Left figure shows how to stabilize the knot negatively while the right figure shows how to stabilize the knot positively.}
\label{Knot_stabilization}
\end{center}
\end{figure}

\subsection{Lefschetz fibrations.}
In this section, we recall basic notions of Lefschetz fibrations. Let $X$ and $B$ be compact oriented smooth manifolds of dimension $4$ and $2$, respectively. A Lefschetz fibration, denoted $f : X \rightarrow B$, is a smooth locally trivial fibration except at a finite set of critical values $\lbrace b_i \rbrace \in B$. If the manifolds have boundary, we can assume that critical values $\lbrace b_i \rbrace \subset int(B)$. In addition, each singular fiber $f^{-1}(b_i)$ has a unique critical point and in a neighbourhood of this point $f$ can be modelled as $f(z_1,z_2) = z_1^2 + z_2^2$. The genus of a regular fiber of $f$ is called the genus of Lefschetz fibration. Singular fibers represent nodal singularities and each singular fiber is obtained by collapsing a simple closed curve (called the vanishing cycle) in a smooth fiber. The monodromy of fibration around the singular fiber is given by positive Dehn twist about the vanishing cycle. The topology of Lefschetz fibration is completely specified by the monodromy representation $\Psi : \pi_1 (B - \{b_i\}) \rightarrow Map(\Sigma)$. If $B = D^2$, then the monodromy along $\partial D^2$ is given by product of positive Dehn twists corresponding to the singular fibers. In this case it is called the \textit{total monodromy} of the fibration. Topology of the Lefschetz fibration can also be described by handle body presentation (see~\cite{Gompf_Stipsicz}) which can be specified by the monodromy representation by attaching $2$-handles to $\Sigma \times D^2$ along the vanishing cycles with framing $-1$ relative to the framing inherited from the fiber. It is a classical result that the monodromy representation, $\Psi$, is uniquely determined upto conjugation by a global diffeomorphism (which is an element of $Map(\Sigma)$) and Hurwitz move which we describe now. Let $\Phi = \tau_1 \tau_2 \dots \tau_k$ be the total monodromy of Lefschetz fibration. Hurwitz move exchanges any two consecutive factors by $\tau_i \tau_{(i+1)} \rightarrow (\tau_{(i+1)})_{\tau_i^{-1}} \tau_i$ or $\tau_i \tau_{(i+1)} \rightarrow \tau_{(i+1)} (\tau_i)_{\tau_{(i+1)}}$, where we use notation $(f)_{\phi} = \phi^{-1} f \phi$.

In the case of Lefschetz fibration over $D^2$ with regular fiber given by $\Sigma_{g,b} , b \neq 0$, the boundary has a natural open book decomposition given by ($\Sigma, \phi$), with $\phi $ being a positive factorization of the total monodromy of the Lefschetz fibration, see~\cite{Akbulut_Ozbagci_Stein_Surface_Lefschetz_fibrations, Loi_Piergallini_Lefschetz_fibrations, Plamenevskaya_Erratum} for more details. The contact structure on the boundary of this Stein structure is isotopic to the one coming from Lefschetz fibration structure. 

\subsection{Spinal Open Books}
\label{spinal_open_books_basics}
A spinal open book is a generalization of the standard open book decomposition where the binding is allowed to be $\coprod_1^n (S^1 \times  \Sigma_i)$ where $\Sigma_i$ can be any surface with boundary. In a standard open book decomposition the binding is $\coprod_1^m (S^1 \times  D^2)$. Note that in a spinal open book the fibers can be non-diffeomorphic surfaces and binding can also be made of non-diffeomorphic surfaces. An abstract spinal open book is given by a $5$-tuple $(M,\hat{F},\hat{\phi},\hat{\Sigma},G)$. Here is $M$ is the $3$-manifold, $\hat{F}$ are the fibers(can be disjoint), $\hat{\phi}$ is orientation preserving diffeomorphism of $\hat{F}$ fixing boundary pointwise, $\hat{\Sigma}$ is the binding, $G$ is a bijection:$|\partial\hat{F}| \cong | \partial \hat{\Sigma}|$. If the contact manifold is fixed we will just denote the supporting spinal open book decomposition by $(\hat{F},\hat{\phi},\hat{\Sigma},G)$. Roughly speaking, spinal open books provide the right contact boundary for Lefschetz fibrations over non disk bases. It is also known that spinal open books under additional restrictions support a unique contact structure. We refer the reader to articles~\cite{Lisi_VHMorris_Wendl_Spinal_Open_Books,Baykur_VHMorris_infinitely_many_euler_characteristic} for details on spinal open books. For proving Theorem~\ref{spinal_finiteness_of_euler_characteristic} we need a version of Wendl's theorem. If the fibers $\hat{F}$ has a planar component, then Wendl's theorem can be generalised.

\begin{theorem}[Lisi-Van Horn-Morris-Wendl, \cite{Lisi_VHMorris_Wendl_Spinal_Open_Books}]
\label{spinal_planar_extension}
If spinal open book $(\hat{F},\hat{\phi},\hat{\Sigma},G)$ has a planar component to $\hat{F}$, then any symplectic filling of the contact manifold $(M,\xi)$ supported by $(\hat{F},\hat{\phi},\hat{\Sigma},G)$ admits a Lefschetz fibration whose boundary is $(\hat{F},\hat{\phi},\hat{\Sigma},G)$.
\end{theorem}

For the purposes of this paper, we will assume that the $\hat{F},\hat{\Sigma}$ are both connected.

\section{Characterization of lantern type relations.}
\label{lantern_relations}
The aim of this section is to give a characterization of the lantern relation. Along with the combinatorial arguments in Section~\ref{combinatorial_arguments} this gives us the ingredients required for the proofs of our theorems on the classification of symplectic fillings of lens spaces and the geography.

We will denote the geometric intersection number of curves and arcs by $i$. For this to be well-defined we assume all curves are isotoped to have minimal intersections.

Recall the classical lantern relation which states states that $\tau_{b_1} \tau_{b_2} \tau_{b_3} \tau_{b_4} = \tau_{\alpha} \tau_{\beta} \tau_{\gamma}$. Here $\alpha, \beta, \gamma$ are curves as shown in Figure~\ref{fig:lantern_normal} and $b_1,\dots,b_4$ denote the curves isotopic to the boundary component as shown.

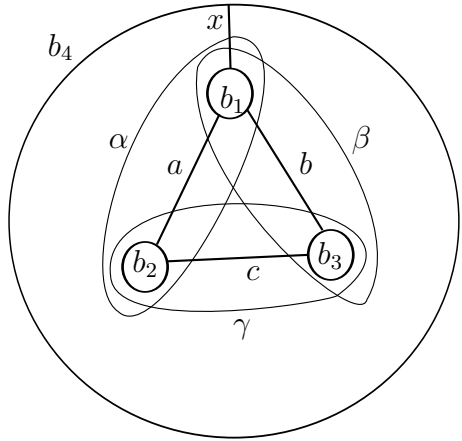
\begin{figure}[htb]
\begin{center}

\begin{overpic}[tics=10]
{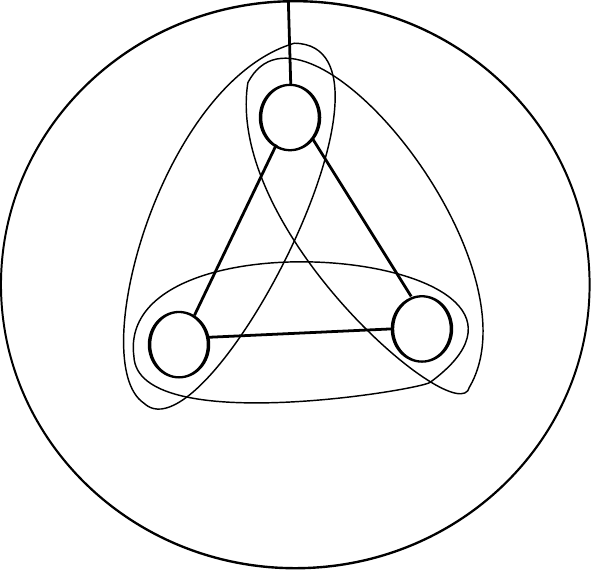}
\put(38, 110){$\alpha$}
\put(130, 110){$\beta$}
\put(85, 40){$\gamma$}
\put(47, 63){$b_2$}
\put(80, 125){$b_1$}
\put(117, 66){$b_3$}
\put(15, 145){$b_4$}
\put(60, 100){$a$}
\put(110, 100){$b$}
\put(90, 60){$c$}
\put(75, 155){$x$}

\end{overpic}
\caption{Classical Lantern relation.}
\label{fig:lantern_normal}
\end{center}
\end{figure}

\begin{lemma}
\label{lantern_normal_inductive_step}
Let $\alpha',\beta' \in \mathbb{D}_3$ be curves that enclose the same set of boundary components as $\alpha, \beta$ respectively and satisfy $\tau_{\alpha} \tau_{\beta} = \tau_{\alpha'} \tau_{\beta'}$. Then there is  $N \in \mathbb{Z}$ such that $\alpha' =  \tau_{\gamma}^N(\alpha)$ and $\beta' = \tau_{\gamma}^N(\beta)$.
\end{lemma}

\begin{proof}
Observe that $\tau_{\alpha}\tau_{\beta} = \tau_{\alpha'}\tau_{\beta'} = \tau_{b_1} \tau_{b_2} \tau_{b_3} \tau_{b_4}\tau_{\gamma}^{-1}$.

We will use the following characterization of multitwists due to Margalit~\cite{Margalit_Lantern_lemma} and independently due to Hamidi-Tehrani~\cite{Hamidi_Tehrani_Lantern}. By a multitwist we mean product of Dehn twists about disjoint curves.

\begin{proposition}

\label{lanetern_lemma}
Let $S$ be any surface and let $\alpha_1$ and $\alpha_2$ be curves in $S$ that intersect minimally and non trivially. If $\tau_{\alpha_1} \tau_{\alpha_2} = M$, where $M$ is a multitwist be a non-trivial relation in $Map(S)$, then the given relation is a lantern relation, that is, a regular neighbourhood $R$ of $\alpha_1 \cup \alpha_2$ is a sphere with $4$ open disks removed from the interior, and $M = \tau_{b_1} \tau_{b_2} \tau_{b_3} \tau_{b_4} \tau_{\alpha_3}^{-1}$. Here $b_1,\dots,b_4$ are curves isotopic to the boundary components of $R$ and $\alpha_3$ is a (non-unique) curve on $R$ with geometric intersection number $2$ with both $\alpha_1$ and $\alpha_2$.

\end{proposition}

Hence we know that $i(\alpha', \beta') = 2 $. Similarly we get the following relations between intersection numbers $ i(\alpha', \gamma) = i(\gamma, \beta') = 2$.

Since curves $\alpha$ and $\alpha'$ are homologous, there exists a diffeomorphism $\phi_1$ which takes the curve $\alpha$ to $\alpha'$. Similarly, there exists a diffeomorphism $\phi_2$ which takes the curve $\beta$ to $\beta'$. We claim that after isotopy the support of each of $\phi_1$ and $\phi_2$ is contained in the subsurface bounded by the curve $\gamma$.	 If $\phi_i$ is not a diffeomorphism supported in the subsurface bounded by $\gamma$, then we will show that each of the curves $\alpha' = \phi_1(\alpha)$ and $\beta' = \phi_2(\beta)$ must intersect curve $\gamma$ at least six times contradicting the computation of intersection numbers above. One way to see this is by thinking of the curves enclosing two different boundary components as represented by an arc joining the two boundary  components. We can impose the condition that these arcs minimize the intersections with the boundary parallel curves corresponding to the boundary components they connect. We have shown the arcs representing the curves $\alpha, \beta$, and $\gamma$ in the Figure~\ref{fig:lantern_normal} and are denoted by $a,b$, and $c$, respectively. In this case, arcs $a$ and $b$ do not intersect the arc $c$. Now if $\phi_1$ were a diffeomorphism not supported in the subsurface bounded by the curve $\gamma$, then $\phi_1(a)$ will intersect the arc $x$ shown in Figure~\ref{fig:lantern_normal} non-trivially. Otherwise we could find a diffeomorphism $\phi_1'$ whose support is contained in the subsurface bounded by the curve $\gamma$ such that $\phi_1'(a)$ is isotopic to $\phi_1(a)$. Since $\phi_1(a)$ intersects $x$ nontrivially, it also intersects the arc $c$ non-trivially. If not then, one can isotope the arc $\phi_1(a)$ to have no intersection with arc $x$. Since $\phi_1(a)$ represented the curve $\alpha'$ this implies that the $i(\alpha',\gamma) \geq 6$.

Since any diffeomorphism which is supported in the subsurface bounded by the curve $\gamma$ is written as a product of Dehn twists given by $\tau_{\gamma}, \tau_{b_2}, \tau_{b_3}$, we get that $\phi_1 = \tau_{\gamma}^{N_1}$ and $\phi_2 = \tau_{\gamma}^{N_2}$. Here we have neglected the boundary Dehn twists $\tau_{b_2}$ and $\tau_{b_2}$ as they act trivially on curves $\alpha$ and $\beta$.

Now to prove the lemma we need to show that $N_1 = N_2$. Towards that end we recall following criterion on intersection numbers (see~\cite{Farb_Margalit_Primer}).

\begin{proposition}
Let $A, B, C$ be any simple closed curves in a surface $S$ and let $n \in \mathbb{Z}$. Then following holds,
\[
\vert n \vert i(A,B) i(A,C) - i(\tau_A^n(C),B) \leq i(B,C)
\]
\end{proposition}

We apply this proposition with curves $A = \gamma, B = \beta, C = \alpha$. Let us assume that $N_1 > N_2$.

Note that $2 = i(\alpha',\beta') = i(\tau_{\gamma}^{N_1}(\alpha), \tau_{\gamma}^{N_2}(\beta)) = i(\tau_{\gamma}^{N_1-N_2}(\alpha), \beta)$. Applying the proposition we get,

\[
 \vert N_1-N_2 \vert i(\gamma, \beta) i(\gamma, \alpha) - i(\tau_{\gamma}^{N_1-N_2}(\alpha), \beta) \leq i(\beta,\alpha) = 2.
\]

\noindent So we see that,
\[ 4|N_1 - N_2| -2 = \vert N_1-N_2 \vert i(\gamma, \beta) i(\gamma, \alpha) - i(\beta,\alpha) \leq i(\tau_{\gamma}^{N_1-N_2}(\alpha), \beta).\]

This gives a contradiction unless $|N_1-N_2| = 0 \text{ or } 1$. Now we are only left to prove that the case $N_1 - N_2 = 1$ cannot happen. First we prove this when $N_1 = 1$ and $N_2 = 0$. This implies that $\beta' \cong \beta$ and $\alpha' \cong \tau_{\gamma}(\alpha)$. In particular, $\alpha' \ncong \alpha$. From the hypothesis we have, $\tau_{\alpha} \tau_{\beta} = \tau_{\alpha'} \tau_{\beta}$ and hence $\tau_{\alpha} = \tau_{\alpha'}$. This in turn implies that $\alpha' \cong \alpha$, which is a contradiction.

Now let us assume that $N_1 = N_2 +1$ and $N_2 \neq 0$. From the hypothesis $\tau_{\alpha} \tau_{\beta} = \tau_{\tau_{\gamma}^{(N_2+1)}(\alpha)} \tau_{\tau_{\gamma}^{N_2}(\beta)} = \tau_{b_1} \tau_{b_2}\tau_{b_3}\tau_{b_4} \tau_{\gamma}^{-1}$. Conjugating by $\tau_{\gamma}^{-N_2}$ on both the sides we see that $\tau_{\tau_{\gamma}(\alpha)} \tau_{\beta} = \tau_{b_1} \tau_{b_2}\tau_{b_3}\tau_{b_4} \tau_{\gamma}^{-1} = \tau_{\alpha} \tau_{\beta}$. Hence we have reduced the problem to the case when $N_1 = 1$ and $N_2 =0$, in which case we already have proved the contradiction. So we get that $N_1 =N_2$.

\end{proof}

Now we prove the uniqueness of curves giving a lantern relation in the following lemma. 

\begin{lemma}
\label{lantern_characterization}
Let $\alpha, \beta, \gamma $ be curves as shown in Figure~\ref{fig:lantern_normal} and $\alpha', \beta', \gamma'$ be curves which enclose the same set of boundary components as $\alpha, \beta, \gamma $, respectively. In addition, suppose that $\tau_{\alpha} \tau_{\beta}\tau_{\gamma} = \tau_{\alpha'} \tau_{\beta'} \tau_{\gamma'}$. Then there exists a diffeomorphism $\psi$ of $\mathbb{D}_3$ such that  $\gamma' \cong \psi(\gamma)$, $\alpha' \cong  \psi (\alpha), \text{ and } \beta'  \cong  \psi (\beta)$.

\end{lemma}

\begin{proof}

Since $\gamma'$ and $\gamma$ enclose the same set of boundary components, there exists a diffeomorphism $\lambda$ taking $\gamma'$ to $\gamma$. Conjugating by $\lambda$ still gives a factorization of  $\tau_{\alpha''}\tau_{\beta''} \tau_{\gamma} = \tau_{\alpha} \tau_{\beta}\tau_{\gamma}$ as $\tau_{b_1} \tau_{b_2} \tau_{b_3} \tau_{b_4}$ commutes with every diffeomorphism. Here, $\alpha'' = \lambda(\alpha)$ and $\beta'' = \lambda (\beta)$. Now we can apply Lemma~\ref{lantern_normal_inductive_step} above to conclude that $\alpha''  \cong \tau_{\gamma}^N (\alpha), \beta''  \cong \tau_{\gamma}^N (\beta)$. Hence, if we let $\psi = \tau_{\gamma}^N \lambda$, we get that $\alpha' = \psi(\alpha), \beta' = \psi(\beta), \gamma' = \psi(\gamma)$. This proves the lemma.

\end{proof}

\section{Combinatorial arguments}
\label{combinatorial_arguments}
In this section we give the combinatorial arguments needed in proofs of our results. Given a diffeomorphism $\Phi=\tau_{\alpha_1} \tau_{\alpha_2} \dots \tau_{\alpha_m}$ of a surface $\mathbb{D}_n$, written as a product of Dehn twists about curves $\alpha_1,\alpha_2,\dots,\alpha_m$. ands another factorization of $\Phi = \tau_{\gamma_1} \tau_{\gamma_2} \dots \tau_{\gamma_k}$ we try to pin down the number of Dehn twists $\tau_{\gamma_i}$ and the boundary components, the curves $\gamma_1,\gamma_2,\dots,\gamma_k$ can enclose.

Before proceeding further we define homomorphisms from $Map(\mathbb{D}_n, \partial \mathbb{D}_n)$ to $\mathbb{Z}$, which define multiplicities associated to Dehn twists. Let  $\Phi \in Map(\mathbb{D}_n, \partial \mathbb{D}_n)$  be a diffeomorphism. Let $b_i$ and $b_j$ be any boundary components of $\mathbb{D}_n$. 

\begin{definition}[\textbf{Joint Multiplicity}] \label{joint_multiplicity}
Capping off all the boundary components of $\mathbb{D}_n$ except $b_i$ and $b_j$ and the outer boundary with disks we obtain a map to $\mathbb{Z} \subset Map(\mathbb{D}_2, \partial \mathbb{D}_2) \cong \mathbb{Z}^3$, which just counts the number of Dehn twists about the curve parallel to the outer boundary. We call this the joint multiplicity of boundary components $b_i$ and $b_j$ and denote it by $M_{ij}(\Phi)$.

\end{definition}

\begin{definition}[\textbf{Mutiplicity}]\label{multiplicity}
Cap off all the boundary components except $b_i$ and the outer boundary. This induces a map from $Map(\mathbb{D}_n, \partial \mathbb{D}_n)$  to $Map(\mathbb{D}_1, \partial \mathbb{D}_1) \cong \mathbb{Z}$ and the map counts the Dehn twists about the boundary parallel curve. We call this the multiplicity of boundary component $b_i$. Denote it by $M_i(\Phi)$.

\end{definition}

\begin{lemma}
\label{first_combinatorial_lemma}
Let $\Phi = \tau_{b_1} \tau_{b_2} \dots \tau_{b_{n+1}}$, where the curve $b_i$ is parallel to the $i^{th}$ boundary component,  be an element of $Map(\mathbb{D}_n.\partial \mathbb{D}_n)$. Let $\Phi'=\tau_{\alpha_1} \tau_{\alpha_2} \dots \tau_{\alpha_m}$ be any other positive factorization of $\Phi$. Then:

\begin{enumerate}
\item If $n \leq 2$ or $n \geq 4$, then $m= n+1$ and each of the curves $\alpha_i$ is a boundary parallel curve $b_j$ for some $j$.
\item If $n = 3$, then either $\Phi'$ is Hurwitz equivalent to  $\tau_{b_1} \tau_{b_2} \dots \tau_{b_{n+1}}$ or Hurwitz equivalent to $ \tau_{\alpha_1} \tau_{\alpha_2} \tau_{\alpha_3}$, where $\alpha_1$ encloses boundary components $b_1$ and $b_2$, $\alpha_2$ encloses boundary components $b_1$ and $b_3$ and $\alpha_3$ encloses boundary components $b_2$ and $b_3$.
\end{enumerate}

\end{lemma}

\begin{proof}
Observe that $M_{ij}(\Phi) = 1$ for all $i, j$ and $M_i(\Phi) =2$ for all $1 \leq i \leq n$.

Let $l_i$ be the number of curves that enclose boundary component $b_i$ and at least one more boundary component in the factorization $\Phi'$ of $\Phi$. Also let $n_j,1 \leq j \leq l_i$ denote the total number of boundary components enclosed by each of these $l_i$ curves. 

It follows from the definition of $M_i$ and the fact that $M_i(\Phi) =2$ that for each $i$, $1 \leq l_i \leq 2$, in any positive factorization of $\Phi$. We will prove that if $l_i=1$ for some $i$, then $l_i = 1 $ for all $i$. To prove this, without loss of generality we can assume that $l_1=1$. This means that there exactly one curve enclosing the boundary component $b_1$ and at least one more boundary component. Let us call the curve $\alpha$. Since $M_{1i}(\Phi) = 1 $ for all $i$, we conclude that $\alpha$ encloses all the boundary components $b_1,b_2,\dots,b_n$. Now by the fact that $M_{ij}(\Phi) = M_{ij}(\Phi') = 1 $ for all $i,j$, we see that $\alpha$ is the unique curve enclosing all the boundary components. In this case all joint multiplicities of all the boundary components are satisfied. Hence, in the factorization of $\Phi$ all other Dehn twists are about boundary parallel curves. Hence we get that, $l_i=1 $ for all $i$. Since all the curves involved are boundary parallel, this factorization is the same as the original.

Now let us assume that $l_i = 2 $ for all $i$. This means there are exactly two curves, $\alpha_1$ and $\alpha_2$, enclosing the boundary component $b_1$. If $i \neq 1$ and the boundary component $b_i$ is enclosed by $\alpha_1$, then the boundary component $b_i$ cannot be enclosed by the curve $\alpha_2$. This is because $M_{1i}(\Phi') = 1$. Hence we see that curves $\alpha_1$ and $\alpha_2$ have only the boundary component $b_1$ in common. Let us assume that curve $\alpha_1$ encloses $k$ boundary components ($k < n$). Without loss of generality, we can assume that $\alpha_1$ encloses boundary components $b_1,b_2,\dots,b_k$ and $\alpha_2$ encloses boundary components $b_1, b_{k+1},\dots,b_n$. Now we make an assumption that $n > 3$. In this case, as $l_2 = 2$ and $M_{2j} = 1 $ for all $j \leq n$, we can conclude that there is a curve $\alpha_3$ enclosing boundary components $b_2,b_{k+1},\dots,b_n$. This contradicts the fact that $M_{k+1,n} = 1$ as in this case $\alpha_2$ and $\alpha_3$ are curves containing both the boundary components. If $k+1 = n$, then instead of the boundary components $b_2$ we apply the same argument to boundary component $b_n$. In this case, $\alpha_3$ will enclose boundary components $b_2,\dots,b_n$ which is a contradiction to the fact that $M_{2(n-1)} = 1$, unless $n=3$. In the case that, $n = 2$ there is nothing to prove as $M_{12} = 1$ would imply that there is a unique curve. This proves part 1 of the lemma.

When $n=3$ and $l_i = 2$, we see that there is a configuration of curves, $\alpha_1$ enclosing boundary components $b_1$ and $b_2$, $\alpha_2$ enclosing boundary components $b_1$ and $b_3$ and curve $\alpha_3$ enclosing boundary components $b_2$ and $b_3$, satisfying the given multiplicities conditions. In this case all the multiplicities, $M_{ij}$ and $M_i$ are satisfied for all $i$ and $j$.

Observe that a priori we do not know the order in which $\tau_{\alpha_1}, \tau_{\alpha_2}$ and $\tau_{\alpha_3}$ appear in the factorization of $\Phi'$. We can always rearrange the terms such that $\Phi'$ is given by $\tau_{\alpha_1}\tau_{\alpha_2} \tau_{\alpha_3}$. For example, if $\Phi' = \tau_{\alpha_2} \tau_{\alpha_1} \tau_{\alpha_3}$, then we can rearrange this as $\Phi' = \tau_{\alpha_1} \tau_{\alpha_1}^{-1} \tau_{\alpha_2} \tau_{\alpha_1} \tau_{\alpha_3} = \tau_{\alpha_1} \tau_{\tau_{\alpha_1}^{-1}(\alpha_2)} \tau_{\alpha_3}$. Noting that conjugating by $\tau_{\alpha_1}^{-1}$ does not change the boundary components enclosed by the curve $\alpha_2$, we still call this new curve $\alpha_2$. This proves part 2 of the lemma.

\end{proof}

\begin{remark}
As seen in the proof above, we can always reorder the elements in factorization upto conjugation since we are only concerned with factorizations up to Hurwitz equvalence. Henceforth, we will assume that the Dehn twists are arranged in the order as in the statement of theorems.  
\end{remark}

Now we give a generalization of this lemma. Here we assume that $n \geq 3$. It is very easy to see, from the proof of the lemma above, that when $n \leq 2$ there is a unique positive factorization of any given element of the mapping class group as every Dehn twist is boundary parallel.

\begin{lemma}
\label{main_combinatorial_lemma}
Assume that $n \geq 3$. Let $\Phi = \tau_{b_1}^r \tau_{b_2} \tau_{b_3} \dots \tau_{b_{n+1}}$, where $r > 1$, be an element of $Map(\mathbb{D}_n.\partial \mathbb{D}_n)$. If $\Phi'$ is any other positive factorization of $\Phi$, then following holds:
\begin{enumerate}
\item If $r \geq n-2$, then the factorization $\Phi'$, up to Hurwitz equivalence, is given by $\tau_{b_1}^r \tau_{b_2} \tau_{b_3} \dots \tau_{b_{n+1}}$ or by the product of following Dehn twists $ \tau_{\alpha_1}, \tau_{\alpha_2}, \dots, \tau_{\alpha_{n-1}}, \tau_{\gamma}, \tau_{b_1}^{(r-n+2)} $ where $\alpha_i$ are curves enclosing boundary components $b_1$ and $b_{i+1}$ only and $\gamma$ is a curve which encloses boundary components $b_2,b_3,\dots,b_n$.
\item If $r < n-2$, then the factorization $\Phi'$, up to Hurwitz equivalence, is given by product of following Dehn twists $ \tau_{b_1}^r, \tau_{b_2}, \tau_{b_3}, \dots, \tau_{b_{n+1}}$.
\end{enumerate}
\end{lemma}

\begin{proof}
We proceed as in the proof of previous lemma. Let $\Phi'= \tau_{\gamma_1} \tau_{\gamma_2} \dots \tau_{\gamma_m}$ be any other positive factorization of $\Phi$. Let $l_i$ be the number of curves enclosing the boundary component $b_i$ and at least one other boundary component, for every $i$. We know from the given factorization that $M_{ij}(\Phi) = M_{ij}(\Phi') = 1 $ for all $i,j$ and $M_i(\Phi) = M_i(\Phi') = 2 $ for all $i > 1$. We get the following set of relations as before 

\begin{center} $ 1 \leq l_i \leq 2 $ for all $ i \geq 2 $ \end{center}
\[ 1 \leq l_1 \leq (r+2) \] 

If $l_i = 1$ for some $i$, then we will prove $l_j = 1 $ for all $j \geq 2$. Let  $\beta_i$ be the unique curve enclosing the boundary component $b_i$. By counting multiplicities (joint and individual) we see that this curve must include all the $n$ boundary components in the disk. Again all the joint multiplicities are satisfied and so only other curves enclosing boundary components $b_j$ for $j\neq i$ are the boundary parallel curves. Hence $l_j = 1$ for all $j$. So we get the unique factorization in this case, by adding the needed boundary Dehn twists. 
	
Now assume that $l_i = 2$. Let $\beta_1, \beta_2$ be the curves which enclose boundary component $b_n$. As argued in the proof of previous lemma, $\beta_1$ and $\beta_2$ have only the boundary component $b_n$ in common. Without loss of generality, we can assume that $\beta_1$ encloses boundary components $b_n,b_1,b_2,\dots,b_k$ and $\beta_2$ encloses boundary components $b_{k+1},\dots,b_{n-1}, b_n$. Here $k < (n-1)$. Let us assume first that $k > 1$. So at least the boundary component $b_2$ is enclosed by $\beta_1$. From the fact that $l_2 = 2$ we get that there has to be a curve, say $\beta_3$ enclosing boundary components $b_2$ and $b_{k+1},\dots,b_{n-1}$. This contradicts the fact that $M_{(k+1)(n-1)} = 1$ except when $(k+1) = (n-1)$. In this case, as $M_{(n-2)1} = 1$, we conclude that either curve $\beta_3$ encloses boundary component $b_1$ also or there is another curve $\beta_4$ enclosing boundary components $b_{n-1}$ and $b_1$. Note that $\beta_4$ can enclose more boundary components than only these two boundary components. In first case we get a contradiction to the fact that $M_{12} = 1$. In the second case, we get a contradiction to the fact that $l_{(n-1)} = 2$.

Now assume that $\beta_1$ encloses boundary components $b_n$ and $b_1$. In this case, from conditions on multiplicities of boundary components and $l_i$, it is easy to see that there has to be curves $\gamma_1, \gamma_2, \dots, \gamma_{n-2}$ such that $\gamma_i$ encloses boundary components $b_1$ and $b_{i+1}$ only. This is possible only if $r+2 \geq n$. This proves the lemma. 

\end{proof}

Refer to Figure~\ref{fig:normal_lantern_combinatorics} for the notations used in following lemma.

\begin{lemma}
\label{combinatorial_lemma}
Let $\alpha$ and $\beta$ be curves as shown in Figure~\ref{fig:normal_lantern_combinatorics}. Assume that $k < n $. Let $$\Phi = \tau_{\alpha} \tau_{\beta} \tau_{b_1}^{r_1} \tau_{b_2}^{r_2} \dots \tau_{b_{(k-1)}}^{r_{(k-1)}} \tau_{b_{(k+1)}}^{r_{(k+1)}} \dots \tau_{b_n}^{r_n}$$ be an element of $Map(\mathbb{D}_n, \partial \mathbb{D}_n)$. Here $r_i \geq 0$ for all $i \neq k$. If $\Phi'$ is any other positive factorization of $\Phi$, then $\Phi'$ is given by product of following Dehn twists $$  \tau_{\alpha'}, \tau_{\beta'}, \tau_{b_1}^{r_1}, \tau_{b_2}^{r_2}, \dots, \tau_{b_{(k-1)}}^{r_{(k-1)}}, \tau_{b_{(k+1)}}^{r_{(k+1)}}, \dots, \tau_{b_n}^{r_n}$$ such that $\alpha', \beta'$ enclose the same set of boundary components as $\alpha, \beta$ respectively.

\end{lemma}

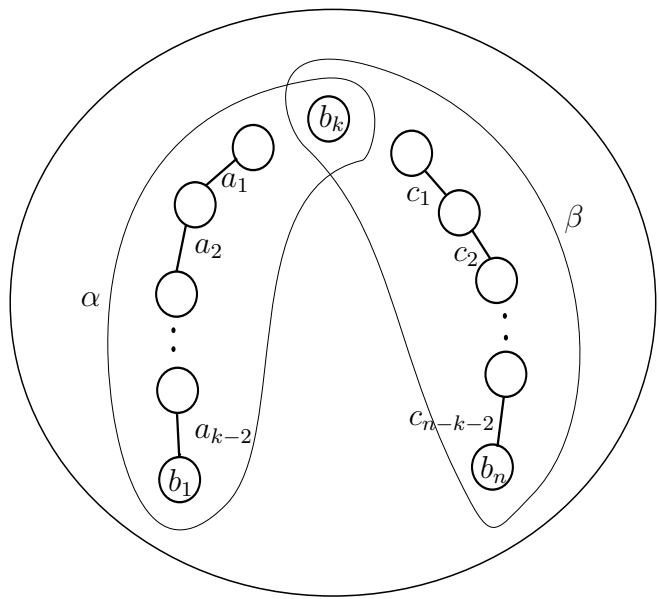
\begin{figure}[ht!]
\begin{center}
\begin{overpic}[tics=10]
{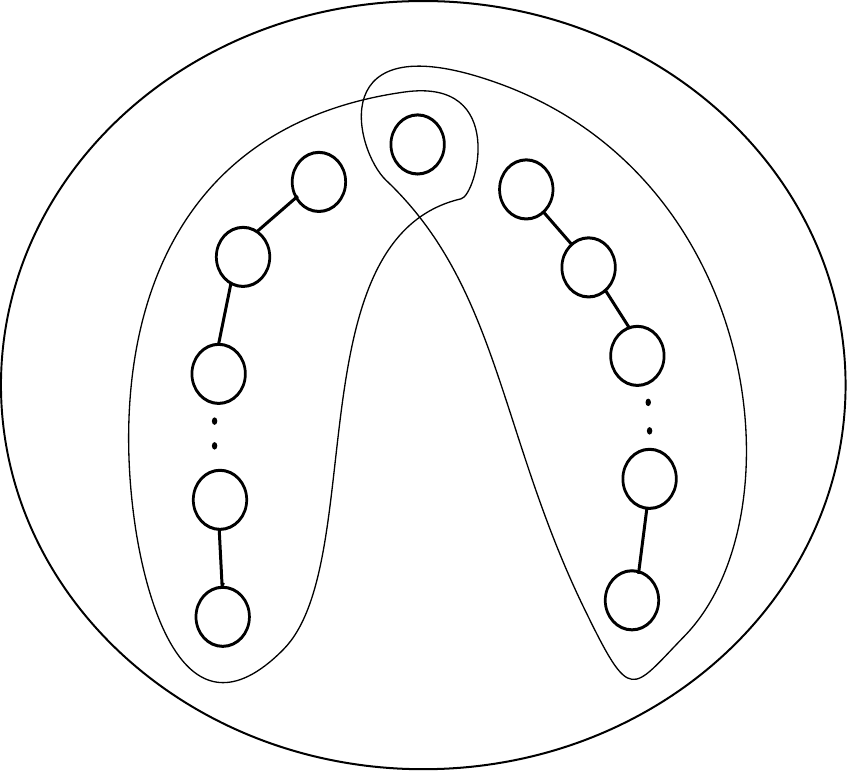}
\put(27, 110){$\alpha$}
\put(210, 140){$\beta$}
\put(60, 40){$b_1$}
\put(117, 178){$b_k$}
\put(178, 45){$b_n$}
\put(80, 155){$a_1$}
\put(70, 130){$a_2$}
\put(70, 60){$a_{k-2}$}
\put(150, 150){$c_1$}
\put(168, 127){$c_2$}
\put(151, 65){$c_{n-k-2}$}

\end{overpic}	

\caption{Figure shows the configuration of curves used in Lemma~\ref{combinatorial_lemma}. The arcs $a_i$ and $b_i$ can be cut along to get a disk with $3$ boundary components in the proof of Theorem~\ref{main_combinatorial_theorem}.}
\label{fig:normal_lantern_combinatorics}
\end{center}
\end{figure}

\begin{proof}
Let $l_i$ be the number of curves enclosing at least $2$ boundary components, each containing boundary component $b_i$. Following the argument given in Lemma~\ref{main_combinatorial_lemma}, we get $ 1 \leq l_i \leq (r_i +1) $ for all $i \neq k$ and $1 \leq l_k \leq 2$. We have the multiplicities $M_{ij} = 1$ for all $i,j$ satisfying $i,j < k$ or $i,j > k$,  $M_{ij} = 0$ for all $ i < k$, $j > k$ and $M_i = r_i + 1$ for all i, and $M_{ki} = 2$ for all $i$. Now we focus on the boundary component $b_k$. There are two cases:

\begin{enumerate}
\item \textbf{$l_k = 1$:} Let us call the curve $\gamma$. Since, $M_{ik} = 1$ for all $i$, we deduce that $\gamma$ encloses all the boundary components in the disk. This is a contradiction as $M_{1n} = 0$.

\item \textbf{$l_k = 2$:} Let us call the curves $\gamma_1$ and $\gamma_2$. Let us assume that $\gamma_1$ encloses boundary components $b_k, b_i, b_j$ such that $i < k$ and $j > k$. This is a contradiction to the fact that $M_{ij} = 0$ as observed before. So we can assume that $\gamma_1$ encloses boundary components $b_k, b_{i_1},\dots, b_{i_r}$, such that $i_j < k$ for all $j \in {1,\dots,r}$. If $\gamma_1$ does not enclose all the boundary components $b_1, \dots, b_k$, we can assume without loss of generality that $\gamma_1$ does not enclose boundary component $b_1$ at least. In this case, $\gamma_2$ will have to enclose boundary components $b_1,b_k, b_{k+1}, \dots, b_n$. A contradiction as $M_{1(k+1)} = 0$. So we get that $\gamma_1$ encloses boundary components $b_1,\dots,b_k$ and similarly $\gamma_2$ encloses boundary components $b_k, \dots, b_n$. 
\end{enumerate}

So we get that any factorization of $\Phi'$ has positive Dehn twists about curves $\alpha'$ and $\beta'$ as in the statement. We are left to prove that $l_i = 1$ for all $i \neq k$. If not, then without loss of generality we can assume that $l_1 \geq 2$. So the boundary components $b_1$ is enclosed by curves $\alpha'$ and at least one more curve, $\gamma$. By assumption $\gamma$ has to enclose at least one more boundary component, say $b_r$, other than $b_1$. It is  clear that $r > k$. But this is a contradiction as $M_{1r} = 0$. So we get that $l_i =1$ for all $ i \neq k$. Now the statement follows by simple count of multiplicities.

\end{proof}

\begin{remark}
This is a generalization of Lemma 2.1 of~\cite{Plamenevskaya_VHMorris_Planar_open_books}. It is also straightforward to see that the lemma holds in general for finitely many curves rather than just $2$ i.e. if $\alpha_1,\dots,\alpha_r$ are curves such that any, $\alpha_i,\alpha_j$ have exactly one boundary component in common (call it $b_k$) for every $i$ and $j$, then any factorization of $\Phi = \tau_{\alpha_1} \tau_{\alpha_2} \dots \tau_{\alpha_r} \tau_{b_1}^{s_1} \dots \tau_{b_{k-1}}^{s_{(k-1)}} \tau_{b_{k+1}}^{s_{(k+1)}} \dots \tau_{b_n}^{s_n}$ is of the form $$\Phi' = \tau_{\alpha_1'} \tau_{\alpha_2'} \dots \tau_{\alpha_r'} \tau_{b_1}^{s_1} \dots \tau_{b_{k-1}}^{s_{(k-1)}} \tau_{b_{k+1}}^{s_{(k+1)}} \dots \tau_{b_n}^{s_n}$$ such that $\alpha_i'$ encloses the same set of boundary components as $\alpha_i$.
\end{remark}

Now we are ready to prove the main theorem, which will be used to classify the Stein fillings of lens spaces in Theorem~\ref{virtually_overtwisted_lens_spaes}. 

\begin{theorem}
\label{main_combinatorial_theorem}
Let $\alpha $ and $\beta$ be the curves as shown in Figure~\ref{fig:normal_lantern_combinatorics}. Let $$\Phi = \tau_{\alpha} \tau_{\beta} \tau_{b_1}^{r_1} \tau_{b_2}^{r_2} \dots \tau_{b_{k-1}}^{r_{(k-1)}} \tau_{b_{k+1}}^{r_{(k+1)}} \dots \tau_{b_n}^{r_n}$$ be a monodromy such that $r_i \geq 1$ for all $i$. Let $\Phi'$ be any other positive factorization of $\Phi$. Then there exists a diffeomorphism $\psi$ such that $\Phi = \psi \Phi' \psi^{-1}$.

\end{theorem}

\begin{proof}

By Lemma~\ref{combinatorial_lemma} above we have that $\Phi' = \tau_{\alpha'} \tau_{\beta'} \tau_{b_1}^{r_1} \tau_{b_2}^{r_2} \dots \tau_{b_{k-1}}^{r_{(k-1)}} \tau_{b_{k+1}}^{r_{(k+1)}} \dots \tau_{b_n}^{r_n}$ such that $\alpha', \beta'$ enclose the same set of boundary components as $\alpha, \beta$ respectively. Since the boundary Dehn twists do not change in any factorization, we need to find all possible choices for $\alpha'$ and $\beta'$ to get all factorizations of $\Phi$. Note that $\tau_{\alpha}\tau_{\beta} = \tau_{\alpha'} \tau_{\beta'}$. Since curves $\alpha$ and $\beta$ do not intersect any of the arcs $a_i, 1 \leq i \leq k-2$ and $c_j, 1 \leq j \leq n-k-2$ which are shown in Figure~\ref{fig:normal_lantern_combinatorics}, we know that $\tau_{\alpha} \tau_{\beta}$ does not move arcs $a_i$ and $c_j$. Hence it follows that $\tau_{\alpha'}\tau_{\beta'}$ does not move them. 

We claim that curves $\alpha'$ and $\beta'$ do not intersect arcs $a_i$ and $c_j$. To see this, we proceed by contradiction. Without loss of generality we can assume that $\beta'$ intersects arc say $a_1$. In this case the arc $a_1$ will be moved strictly to the right by $\tau_{\beta'}$. This follows from the fact that any positive factorization is right veering, see~\cite{HKM_Right_Veering_1} for details. Since $\tau_{\alpha'} \tau_{\beta'}$ does not move the arc $a_1$, it will have to be moved left by the other factor $\tau_{\alpha'}$ in $\tau_{\alpha'} \tau_{\beta'}$. This is not possible as every factor is positive and hence right veering. So we get a contradiction. Similarly we can prove that $\alpha'$ and $\beta'$ do not interest any of the arcs $a_i$ and $c_j$. Hence, $\alpha'$ and $\beta'$ live in the complement of arcs $a_i$ and $c_j$. So we can cut the surface along $a_i$ and $c_j$ to specify $\alpha'$ and $\beta'$. Now the result follows from lantern characterization in Lemma~	\ref{lantern_normal_inductive_step}. 

\end{proof}

Using this theorem, we can prove the classification of the Stein fillings of virtually overtwisted contact structures on $L(p,1)$ due to Plamenevskaya and Van Horn-Morris. In addition we can also reprove the classification of the Stein fillings on universally tight $L(p,1)$ due to McDuff~\cite{McDuff_Rational_ruled_surfaces}.

\begin{corollary}
\label{Corollay_McDuff_Classification}
Let $\xi$ be any tight contact structure on $L(p,1)$. Then

\begin{enumerate}
\item The contact structure $\xi$ has a unique Stein filling if $p \neq 4$ upto symplectomorphism. 
\item The universally tight contact structure on $L(4,1)$ has exactly two Stein fillings upto symplectomorphism. 
\item The virtually overtwisted contact structure on $L(4,1)$ has a unique Stein filling upto symplectomorphism.
\end{enumerate}
\end{corollary}

\begin{proof}
First consider the case when the contact structure is virtually overtwisted. For this we draw the open book decomposition as shown in Figure~\ref{fig:open_book_supporting_lens_space}. We describe the open book  first. The left picture shows annulus open book supporting $(S^3,\xi_{std})$ where the dotted curve $\alpha$ is the Dehn twist curve. The solid curve is the Legendrian unknot $tb=-1$, sitting on the page of this open books. Now one can stabilize this unknot $p$ times as shown in the right by curve $\beta$. Dotted curves are the stabilization curves which intersect the co-cores on the $1$-handles attached exactly once. By isotoping the whole surface we see that this is exactly the surface described by Figure~\ref{fig:normal_lantern_combinatorics}. Now it is easy to see that, monodromy for the contact structures on the lens spaces is given by $$\Phi = \tau_{\alpha} \tau_{\beta} \tau_{b_1} \tau_{b_2} \dots \tau_{b_{k-1}} \tau_{b_{k+1}} \dots \tau_{b_n}$$ where $\alpha$ and $\beta$ are curves as shown in Figure~\ref{fig:normal_lantern_combinatorics} with $n = (p-1)$. Hence the last statement of the theorem follows from Theorem~\ref{main_combinatorial_theorem}.

Now consider the case when the contact structure is universally tight on $L(p,1)$. In this case monodromy is given by $\Phi = \tau_{b_1} \tau_{b_2} \dots \tau_{b_n} \tau_{b_{n+1}}$ and $n = (p-1)$. By Lemma~\ref{first_combinatorial_lemma}, if $\Phi'$ is any positive factorization of $\Phi$, then $\Phi ' $ is $ \Phi$ except when $p =4$. 

When $\xi$ is a universally tight contact structure on $L(4,1)$. Then we know that either $\Phi = \tau_{b_1} \tau_{b_2} \tau_{b_3} \tau_{b_4}$ which by lantern relation is same as $\tau_{\alpha} \tau_{\beta} \tau_{\gamma}$, with $\alpha,\beta,\gamma$ as shown in Figure~\ref{fig:lantern_normal}. By Lemma~\ref{first_combinatorial_lemma} we know that any other factorization of $\Phi$ is of the form $ \tau_{\alpha'} \tau_{\beta'} \tau_{\gamma'}$ where $\alpha'$ is a curve that encloses boundary components $b_1$ and $b_2$, $\beta'$ is a curve that encloses boundary components $b_1$ and $b_3$ and $\gamma'$ is a curve that encloses boundary components $b_2$ and $b_3$. In the second case the lantern charazterization given in Lemma~\ref{lantern_characterization}, implies that there exists a diffeomorphism $\psi \in Map(\mathbb{D}_3,\partial \mathbb{D}_3)$ such that $\alpha' = \psi(\alpha), \beta' = \psi(\beta), \gamma' = \psi(\gamma)$. Hence, we get that there are exactly two factorizations of $\Phi$ upto diffeomorphism and exactly two Stein fillings upto symplectomorphism.

\end{proof}

After proving the known results using our techniques, we classify the Stein fillings of any contact structure on $L(p(m+1)+1,(m+1))$.

\begin{comment}
\begin{figure}[ht]
\begin{center}
\includegraphics[scale=0.25]{legendrian_surgery_L(2p-1,2).pdf}
\caption{Legendrian surgery picture for tight contact structures on $L(2p-1,2)$.}
\end{center}
\end{figure}
\end{comment}

\begin{proof}[Proof of Theorem~\ref{virtually_overtwisted_lens_spaes}]

From the classification given in~\cite{Honda_Classification_1}, we can draw the Legendrian surgery diagram for various contact structures on $L(p(m+1)+1,(m+1))$ as shown in Figure~\ref{fig:legendrian_surgery_L(2p-1,2)}. Now we draw the open book decomposition corresponding to this contact manifold. This is exactly the same as described in Corollary~\ref{Corollay_McDuff_Classification}. Figure~\ref{fig:open_book_supporting_lens_space} shows the open book. The boundary parallel curves which are not dotted are $m$ Legendrian unknots with $tb =-1$. We see that this is exactly the surface described by Figure~\ref{fig:normal_lantern_combinatorics}. 

When $r=1$, the contact structure is universally tight. For the universally tight contact structure the monodromy is $\Phi = \tau_{b_1}^{m+1} \tau_{b_2} \dots \tau_{b_n}$ and the page of the open book decomposition is $\mathbb{D}_n$ with $n = p-1$. By Lemma~\ref{main_combinatorial_lemma} above, we have that any other factorization $\Phi' $ is $  \tau_{b_1}^{m+1} \tau_{b_2} \dots \tau_{b_n}$ when $ p > m+4$. Hence, we get the uniqueness of the Stein filling as the factorization is unique. When $p \leq m+4$, we have by Lemma~\ref{main_combinatorial_lemma}, that there is another factorization $\Phi ' =  \tau_{\alpha_1'} \tau_{\alpha_2'} \dots \tau_{\alpha_{p-2}'} \tau_{\gamma'} \tau_{b_1}^{(r-p+3)}$, where $\alpha_i'$ are curves enclosing boundary components $b_1$ and $b_i$ for every $i$ and $\gamma'$ is a curve enclosing boundary components $b_2,b_3,\dots,b_{p-1}$. So we have at least $2$ Stein fillings. This finishes the proof.

The open book decomposition for virtually overtwisted contact structures is given by Figure~\ref{fig:open_book_supporting_lens_space} with monodromy $\Phi = \tau_{\alpha} \tau_{\beta} \tau_{b_1} \tau_{b_2} \dots \tau_{b_{k-1}} \tau_{b_{k+1}} \dots \tau_{b_n}^{m+1}$, where $n = p-1$ and $\alpha$ and $\beta$ are curves as shown in Figure~\ref{fig:open_book_supporting_lens_space}. In this case, the uniqueness of the Stein filling follows exactly in the same way as above.

\begin{figure}[ht!]
\begin{center}
\begin{overpic}[tics=10]
{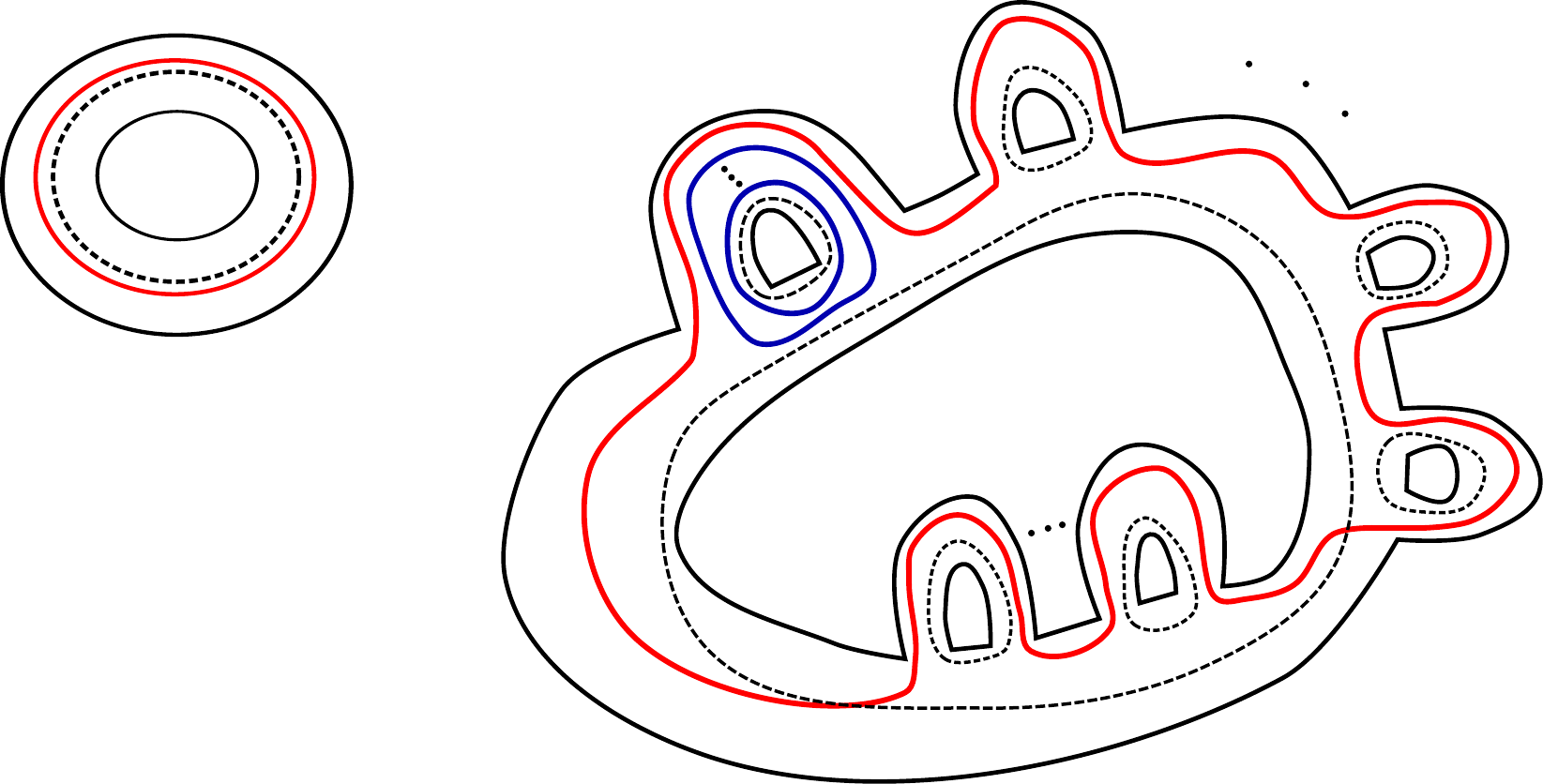}
\put(40,210){\textbf{$\alpha$}}
\put(170, 60){\textbf{$\beta$}}
\put(400,160){\textbf{$\alpha$}}
\put(230,188){\textbf{$m$}}

\end{overpic}	
\end{center}
\caption{Open book for $(S^3,\xi_{std})$ on left and open book supporting the lens space $L(p(m+1)+1,(m+1))$ is shown.}
\label{fig:open_book_supporting_lens_space}
\end{figure}
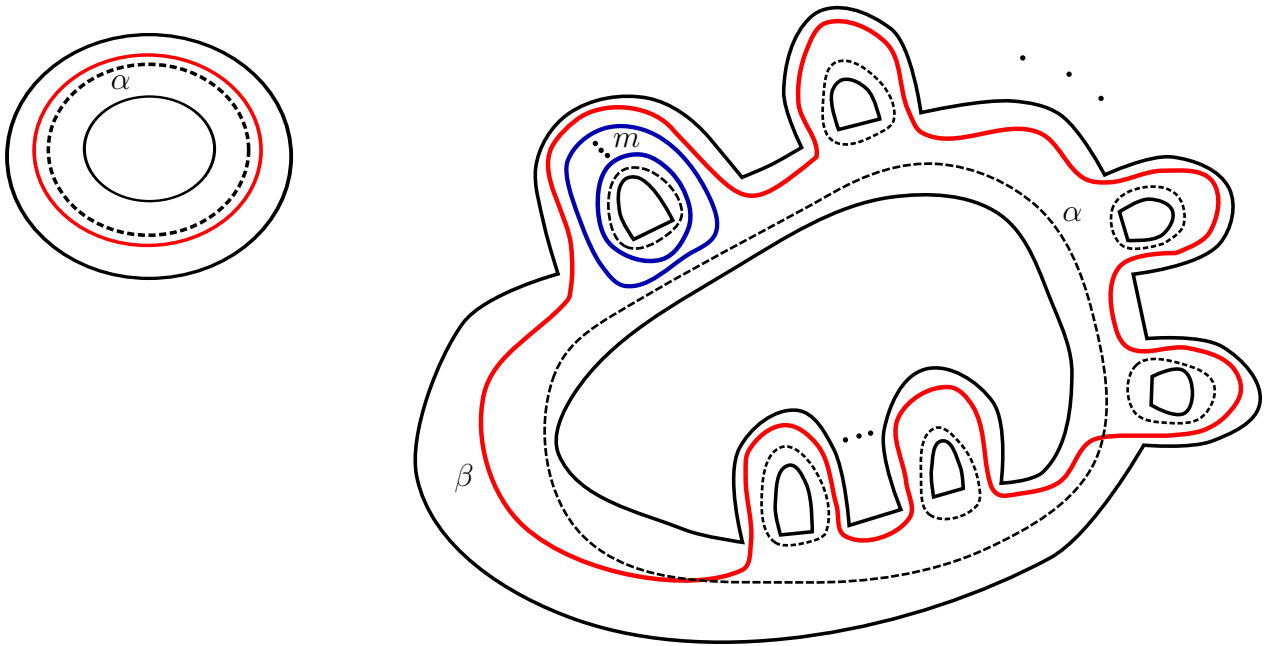
\end{proof}

\section{Finiteness of Euler Characteristics and Signature}
\label{euler_characteristic_signature}
In this section we prove Theorem~\ref{Finite_euler_characteristic} and Theorem~\ref{spinal_finiteness_of_euler_characteristic}.

\begin{proof} [Proof of Theorem~\ref{Finite_euler_characteristic}]
Let $(X,J)$ be a Stein filling of a contact manifold $(M,\xi)$ supported by planar open book $(\mathbf{D_n}, \Phi)$. It follows from Wendl's theorem that the handle decomposition for $X$ is given by one $0$-handle, $n$ $1$-handles and $2$-handles. The number of $2$-handles is given by the number of vanishing cycles. Hence if we bound the number of vanishing cycles we will have a bound on the Euler characteristic of the given Stein filling. In other words, if we can bound the number of Dehn twists in any positive factorization of $\Phi$ we get an upper bound on the Euler characteristic of $(X,J)$.

Let $M_i$ denote the multiplicity of each boundary component, as defined in Section~\ref{combinatorial_arguments}, in this factorization. By definition of multiplicity of a boundary component, in any factorization $\Phi'$ of $\Phi$ there cannot be more than $M_i$ positive Dehn twists about curves enclosing boundary component $b_i$. Hence, $M_1 + M_2 + \dots + M_n$ gives an upper bound on the number of two handles attached. This gives an upper bound on $\chi(X)$. 

Now we are left to bound the signature of the Stein filling. Recall that Euler characteristic of a Stein manifold $X$ can be written as $\chi(X) = 1 - b_1(X) + b_2^+(X) + b_2^-(X) + b_2^0(X)$. Using Theorem $4.1$ of~\cite{Etnyre_Planar_open_books}, which states that for a manifold supported by planar open book any Stein filling has vanishing $b_2^+$ and $b_2^0$, we get that $\chi(X) = 1- b_1(X) + b_2^-(X)$. Hence, $\sigma(X) + \chi(X) = 1 - b_1(X)$. Now for the $4$-manifold $X$, a simple homology computation shows that $|H_1(X)| \leq n$ and so $1-b_1(X)$ is bounded. It follows that $|\sigma(X) + \chi(X)| < M$ for some $M$. Now by finiteness of $\chi(X)$ we get that $|\sigma(X)| < M + |\chi(X)|$. Hence, there exists a constant $N$ such that $|\sigma(X)| < N$ and $|\chi(X)| < N$. It follows that $\mathcal{C}_{(M,\xi)}$ is finite.
\end{proof}

\begin{remark}
After proving the above theorem, the author found another proof of this result. Stipsicz~\cite{Stipsicz_Geoegraphy_of_stein_fillings} proved that $\mathcal{C}_{(M,\xi)}$ is finite for any manifold $(M,\xi)$ such that every Stein filling has $b_2^+ = 0$. Now the above mentioned theorem of Etnyre~\cite{Etnyre_Planar_open_books} implies that for any manifold supported by planar open book $b_2^+ = 0$. Hence, the theorem follows by combining these two results.

\end{remark}

Before we prove the 	finiteness results for spinal open books, we recall a theorem of Wand~\cite{Wand_Planar_Open_Books_Euler_Characteristic} and state it in a more general form suitable for applications to our purposes. The proof essentially is the same as given by Wand in~\cite{Wand_Planar_Open_Books_Euler_Characteristic} but we give a sketch here for completeness.

\begin{proposition}
\label{Wand_generalization_to_spinal_planar}
Let $(M,\xi)$ be a contact manifold supported by spinal open book with connected binding $\Sigma_{g,r}$ and connected fibers $\Sigma_{0,b}$. If $(X_1,J_1)$ and $(X_2,J_2)$ are any two Stein fillings of $(M,\xi)$, then $\chi(X_1) + \sigma(X_1) = \chi(X_2) + \sigma(X_2).$
\end{proposition}

\begin{proof}
Let us denote by $X_{g,\lambda}^{h}$, a Lefschetz fibration over a closed surface $\Sigma_h$ with fibers $\Sigma_g$ and $\lambda$ is factorization of identity in $Map(\Sigma_h, \partial \Sigma_h)$. If $\lambda'$ is obtained from $\lambda$ by $r$-substitution ($r$ is a relator in the mapping class group of the surface $\Sigma_g$), then a result of Endo and Nagami~\cite{Endo_Nagami_Signature_of_relators} gives:

\[\sigma(X_{g,\lambda}^{h}) - \sigma(X_{g,\lambda'}^{h}) = I(r)\]

\noindent where $I(r)$ is the signature of the relator $r$ defined in~\cite{Endo_Nagami_Signature_of_relators}. 

With this set-up we can prove following statement which is essentially Theorem 4.4 of ~\cite{Wand_Planar_Open_Books_Euler_Characteristic}. 

\begin{lemma}

Let $\Sigma = \Sigma_{g,b}$ be a surface with boundary. Let $X_{\Sigma,\lambda}$ and $X_{\Sigma,\lambda'}$ be Lefschetz fibrations over $\Sigma_{h,m}$ such that $\lambda'$ is an $r$-substitution of $\lambda$. Then

\[\sigma(X_{\Sigma,\lambda}) - \sigma(X_{\Sigma,\lambda'}) = I(r).\]

\end{lemma} 

Starting with a Lefschetz fibration over non-closed surface $\Sigma_{h,m}$ we can construct the closed Lefschetz fibrations $X_{\bar{g},\lambda_1}^{h}$ and $X_{\bar{g},\lambda_1'}^{h}$ with genus of fibers $\bar{g} > g$, such that $\lambda_1'$ is obtained by an $r$-substitution from $\lambda_1$ in mapping class group of $\Sigma_{\bar{g}}$. To see this, $\Sigma_{0,b+1}$ be a sphere with $1$ more boundary component than the fibers $\Sigma_{g,b}$. Let $\Sigma'' = \Sigma_{g,b} \cup \Sigma_{0,b+1}$ where we glue the boundary components and let $\hat{\Sigma}$ denote the surface obtained by capping off the boundary component of $\Sigma''$ with a disk. As $\Sigma_{g,b}$ is a subsurface of $\Sigma''$ extending by identity on $\Sigma'' \backslash \Sigma_{g,b}$, both $\lambda$ and $\lambda'$ extend to mapping classes on $\Sigma''$. We still call the extensions $\lambda$ and $\lambda'$ in the new surface. These extensions are also related by $r$-substitution in the new mapping class group. It is a well known fact that, any positive mapping class $\Phi$ in a genus $p$ surface with $1$ boundary component can be written as $\Phi = \tau_{\delta}^N \hat{\Phi}$, where $\tau_{\delta}$ is Dehn twist about the boundary component and $N > 0 $ and $\hat{\Phi}$ is a negative mapping class, that is, given as factorization in terms of negative Dehn twists only. Applying this fact gives us $\lambda = \tau_{\delta}^N \hat{\lambda}$ and so $\lambda \circ \hat{\lambda}^{-1} = \tau_{\delta}^N$.  Hence, $\lambda_1 = \lambda \circ \hat{\lambda}^{-1}$ gives a positive factorization of identity in $Map(\hat{\Sigma})$ and hence a Lefschetz fibration $X_{\bar{g},\lambda_1}^{h}$. Similarly, $\lambda_1'= \lambda' \circ \hat{\lambda}^{-1}$ gives a Lefschetz fibration $X_{\bar{g},\lambda_1'}^{h}$. We note that, in getting the closed Lefschetz fibrations $X_{\bar{g},\lambda_1}^{h}$ and $X_{\bar{g},\lambda_1'}^{h}$, we have added the same compact $4$ manifold $Y$. We have, 	  

\[\sigma(X_{\bar{g},\lambda_1}^{h}) - \sigma(X_{\bar{g},\lambda_1'}^{h}) = I(r).\]

\noindent  So by Novikov additivity we have, $ I(r) = \sigma(X_{\bar{g},\lambda_1}^{h}) - \sigma(X_{\bar{g},\lambda_1'}^{h}) = \sigma(X_{\Sigma,\lambda}) + \sigma(Y)- \sigma(X_{\Sigma,\lambda'}) -\sigma(Y)$ and the result follows.

To get the result for spinal open books with planar fibers, we apply Theorem~\ref{spinal_planar_extension}, and note that any relator in a planar surface is a concatenation of lantern relator. Thus $I(r) =0$ and the result follows. See~\cite{Endo_Nagami_Signature_of_relators} for calculations of signatures of various relators in mapping class groups.

\end{proof}

Now we are ready to prove the finiteness of the Euler characteristic and the signature of the Stein filling of spinal open books. Just as in the proof of Theorem~\ref{Finite_euler_characteristic}, we will bound the number of $2$-handles to get an upper bound on the Euler characteristic of Stein filling $(X,J)$. 

\begin{proof}[Proof of Theorem~\ref{spinal_finiteness_of_euler_characteristic}]

The monodromy of a Lefschetz fibration over a genus $g$ surface with fiber $\Sigma_h$ is given by $$w = \prod_{i=1}^m \tau_{v_i} \prod_{j=1}^g [\alpha_j,\beta_j]$$ where $\alpha_1,\dots,\alpha_g,\beta_1,\dots,\beta_g$ are images of generators of fundamental group in $Map(\Sigma_h,\partial \Sigma_h)$  and  $\tau_{v_1},\dots, \tau_{v_m}$ are Dehn twists about vanishing cycles. To construct this manifold, we start with a surface bundle over the $g$ surface. This is described by the monodromy $\prod_{j=1}^g [\alpha_j,\beta_j]$. This manifold has a finite Euler characteristic independent of $\alpha_j$ and $\beta_j$. To get the Lefschetz fibration we attach $2$-handles along the vanishing cycles prescribed by curves $v_i$. Hence to bound the Euler characteristic again we need to bound the number of vanishing cycles. In the abelianization of the mapping class group of a planar surface $w$ has an image $\bar{\tau_{v_1}},\dots, \bar{\tau_{v_m}}$, where $\bar{\tau_{v_i}}$ is image of the Dehn twist in the abelianization. So to bound the Euler characteristic of the Stein filling, it suffices to bound the number of terms in the factorization $\bar{\tau_{v_1}},\dots, \bar{\tau_{v_m}}$. This is done in exactly the same way as in the proof of Theorem~\ref{Finite_euler_characteristic}. So we get the finiteness of $\chi(X)$. Now from Proposition~\ref{Wand_generalization_to_spinal_planar} we see that $\sigma(X)$ is also bounded.

\end{proof}

\section{Euler characteristic of sphere plumbings}
\label{Section:Sphere_Plumbings}
In this section we prove Theorem~\ref{Thm:Sphere_Plumbing}. Gay and Mark~\cite{Gay_Mark_Convex_Configuration} explicitly write down the open book decomposition for the boundary of a plumbing of spheres. The open book decomposition for $(M,\xi_{pl})$ which is contactomorphic to boundary of $(Z,\eta)$ which is a neighbourhood of spheres $C = C_1 \cup C_2 \cup \dots \cup C_n$, intersecting $\omega$-orthogonally, along the negative definite graph $\Gamma$ is given as follows. Recall from Section~\ref{introduction}, that we assume the row sum satisfies $s_i = \sum_j q_{ij} \leq 0$, where $Q = (q_{ij} = [C_i] \cdot [C_j]$ is the intersection matrix. Let $S$ be the result of connect summing $|s_i|$ copies of $D^2$ to each $C_i$ and then connect summing these surfaces according to $\Gamma$. It is clear from this construction that $S$ is a disk with a finitely many open disks removed from the interior when $C_i$ are all spheres. Let $\lbrace c_1,c_2,\dots,c_k \rbrace$ be the collection of simple closed curves on $S$ consisting of one curve around each connect sum neck. It is clear from the construction that $c_i$ are all disjoint. Let $\tau$ denote the product of Dehn twists along these curves. The following theorem is proved in~\cite{Gay_Mark_Convex_Configuration}.

\begin{theorem}
\label{Thm:Gay_Mark}
Any neighbourhood of $C$ contains a neighbourhood $(Z,\eta)$ of $C$ 	with strongly convex boundary, that admits a symplectic Lefschetz fibration $\pi: Z \rightarrow D^2$ having regular fiber $S$ and exactly one singular fiber $S_0 = \pi^{-1}(0)$. The vanishing cycles for $\pi$ are $c_1,\dots,c_k$ and the induced contact structure $\xi_{pl}$ on $\partial Z$ is supported by the induced $(S,\tau)$.
\end{theorem}

To prove the Theorem~\ref{Thm:Sphere_Plumbing}, we first prove the following fact about positive factorizations in planar mapping class group. 

\begin{theorem}

Assume that $\phi \in Map(\mathbb{D}_n,\partial \mathbb{D}_n)$ can be written as a product of positive Dehn twists about disjoint curves, i.e. $\phi = \tau_{c_1} \tau_{c_2} \dots \tau_{c_k}$, such that $c_i$ are all disjoint. If $\tau_{d_1} \tau_{d_2} \dots \tau_{d_m}$  is any other positive factorization of $\phi$, then $m \leq k$. 

\end{theorem}

\begin{proof}
The proof is by induction on the number of holes, $n$, of the disk and the number of curves $k$. 

We start by proving the base cases. There are two base cases to be checked.

\begin{itemize}

\item \textbf{Disk with $1$ hole:} Let us denote the hole by $b_1$. Now $\Phi = \tau_{b_1}^p$ for some $p \in \mathbb{Z}_{\geq 0}$. Any other positive factorization has to be Dehn twists about the hole $b_1$. So the argument is trivial in this case.

\item  \textbf{Disk with $n>1$ holes and $k=1$:} In this case $\phi = \tau_{\alpha}$ for some curve $\alpha$. We can assume that the curve $\alpha$ encloses the boundary components $b_1,b_2,\dots,b_l$ for some $l \leq n$. If $l = n$ then the curve $\alpha$ is boundary parallel to the outer boundary component $b_{n+1}$ and so a positive factorization of $\phi$ is unique. So we can assume $ l < n$. In this case the joint multiplicity of $M_{ij} = 0 $ for $i \leq l$ and $j > l$. So a positive factorization of $\phi$ cannot include Dehn twists about curves enclosing any of the holes $b_{l+1},\dots, b_n$. Also since the multiplcity $M_i = 1$ of each holes $b_1,\dots,b_l$, we know that there can be no more than one curve enclosing holes any of the holes $b_1,\dots, b_l$ in any positive factorization of $\phi$. In addition, since the joint multiplicity of all these holes is $1$, so there must be exactly $1$ curve enclosing the holes $b_1,\dots,b_l$ in any positive factorization of $\phi$. 

\end{itemize}

Now by induction we assume that the theorem is true for:
\begin{enumerate}
 \item All planar surfaces with $n-1$ holes.
 \item Any $\phi$ in $Map(\mathbb{D}_n,\partial \mathbb{D}_n)$, with a positive factorization such that the Dehn twists are about $k-1$ disjoint curves in the surface $\mathbb{D}_n$.
 
 \end{enumerate}
 
In the surface $\mathbb{D}_n$, let $\phi = \tau_{c_1} \tau_{c_2} \dots \tau_{c_k}$ be an element of the mapping class group such that $c_i$'s are all disjoint. Let us assume that there is a positive factorization of $\phi$ given by $\tau_{d_1}\tau_{d_2} \dots \tau_{d_m}$ with $m > k$. If there is any hole $b_i$ with multiplicity $0$, then we can cap off the hole $b_i$ and this gives a contradiction to the induction hypothesis~$1$. So we can assume that multiplicity of each of the holes $b_1,\dots,b_{n}$ is at least~$1$. Let us start by looking at the hole~$b_1$. If there is a boundary parallel Dehn twist $\tau_{b_1}$ in both the factorizations of $\phi$ given above, then we can cancel the boundary Dehn twist and get a contradiction to the induction hypothesis~$2$. So $\tau_{b_1}$ can never appear in both the factorizations of $\phi$. Let us assume that $\tau_{b_1}$ appears in the factorization $\tau_{c_1} \tau_{c_2} \dots \tau_{c_k}$. In this case we can cap off the hole $b_1$ and get a contradiction to the induction hypothesis $1$, as capping off the hole $b_1$ reduces the number of factors in $\tau_{c_1} \tau_{c_2} \dots \tau_{c_k}$ by $1$, but does not reduce the number of factors in $\tau_{d_1}\tau_{d_2} \dots \tau_{d_m}$. Now we argue that there is at least one factor of $\tau_{b_1}$ in the positive factorization $\tau_{d_1}\tau_{d_2} \dots \tau_{d_m}$. If not then by capping off the hole $b_1$ we will not reduce the number of factors in both the factorizations of $\phi$. Again a contradiction to the induction hypothesis $1$. The same exact argument holds for each of the holes $b_1,\dots ,b_{n+1}$. So we know that there is at least $1$ factor of $\tau_{b_i}$ for each $i \in \lbrace 1, \dots, n+1 \rbrace$ in the factorization $\tau_{d_1}\tau_{d_2} \dots \tau_{d_m}$. This implies that $M_{ij} \geq 1$ for every $1 \leq i,j \leq n$ since there is a Dehn twist about the outer boundary component $b_{n+1}$ in $\tau_{d_1}\tau_{d_2} \dots \tau_{d_m}$.

Now since there are no Dehn twists about boundary parallel curves in $\tau_{c_1}\tau_{c_2} \dots \tau_{c_k}$, we know that there is a curve $\alpha$ in $c_1,\dots,c_k$ which encloses a proper subset of holes. Without loss of generality we can assume that $\alpha$ enclosed holes $b_1,\dots, b_l$ for some $l < n$. Since the curves $c_1,\dots, c_k$ are disjoint there is no curve enclosing any of the holes $b_1,\dots, b_l$ and at least one of the holes $b_{l+1},\dots,b_{n}$. If there is such a curve $\beta$, then $\beta$ will intersect $\alpha$ non trivially, which is not possible. So we can find at least two holes, amongst $b_1,\dots,b_n$, which are never enclosed together by any of the curves $c_1,\dots,c_k$. So there is a pair of holes, call them $b_{l_1}$ and $b_{l_2}$, such that $M_{l_1 l_2} = 0$. A contradiction to the fact that $M_{ij} \geq 1$ for every $1 \leq i,j \leq n$ observed above. 
\end{proof}

\begin{proof}[Proof of Theorem~\ref{Thm:Sphere_Plumbing}]

 If $(X,J)$ is any other strong symplectic filling of $(M,\xi_{pl})$, then it has an open book decomposition with page $S$ constructed above and the monodromy which is a positive factorization $\tau_{d_1}\tau_{d_2} \dots \tau_{d_m}$ of $\tau$. So we conclude that, $\chi(X) = 1 + (n-1) + m$. By the construction of $Z$ as described in the beginning of this section, we know that $\chi(Z) = 1 + (n-1) + k$. We know from the theorem above that $m \leq k$.  The proof follows easily.    

\end{proof}

\section{A Curious Example}

\begin{proof}[Proof of Theorem~\ref{curious_example}]
Let us consider a particular word in $Map(\mathbf{D_4}, \partial \mathbf{D_4})$. Let $a,b,c,d,\gamma$ be the curves shown in Figure~\ref{curious_example_fig}. Let $\Phi = \tau_a \tau_b \tau_c \tau_{b_1} \tau_{b_3} \tau_{b_4}$. We will find infinite factorizations of this word. Towards that end, note that $\tau_{b_1} \tau_{b_2} \tau_{b_3} \tau_{\gamma} \tau_d ^{-1} = \tau_a \tau_b $. Now conjugating by $\tau_d^{n}$ for $n \in \mathbb{Z}$, we get $\tau_{\tau_d^n(a)} \tau_{\tau_d^n(b)} = \tau_{b_1} \tau_{b_2} \tau_{b_3} \tau_{\gamma} \tau_d ^{-1}$. Hence we get $\tau_{\tau_d^n(a)} \tau_{\tau_d^n(b)} = \tau_a \tau_b$. So we get $\tau_{\tau_d^n(a)} \tau_{\tau_d^n(b)} \tau_c \tau_{b_1} \tau_{b_3} \tau_{b_4} = \tau_a \tau_b \tau_c \tau_{b_1} \tau_{b_3} \tau_{b_4}$. There can be no diffeomorphism which takes the left hand side of the above expression to the right hand side. If there is a diffeomorphism $\psi$ such that $\psi \tau_{\tau_d^n(a)} \tau_{\tau_d^n(b)} \tau_c \tau_{b_1} \tau_{b_3} \tau_{b_4} \psi^{-1}$ is $\tau_a \tau_b \tau_c \tau_{b_1} \tau_{b_3} \tau_{b_4}$, then we claim that this would imply that $\psi = \tau_d^{-n}$. This is a contradiction as in this case $\psi(c)$ is not isotopic to $ c$. 

To prove the claim, we proceed by contradiction. Assume that $\psi \neq \tau_{d}^{-n}$. Since $\psi$ does not move the curve $c$, it is easy to see that it does not move the arc $\alpha$ as shown in Figure~\ref{curious_example_fig}. Hence we get that the diffeomorphism $\psi$ is supported in the subsurface bounded by the curve $c$. This is impossible as $\psi(\tau_d^n(a))$ is isotopic to $a$ and $\psi(\tau_d^n(b))$ is isotopic to $b$. 

Hence we get infinitely many factorizations of the same monodromy. So potentially there are infinitely many Stein fillings of this manifold.

We can write down the surgery picture for the Stein filling. Observe that $(\mathbf{D}_4, \tau_a \tau_b \tau_{b_1} \tau_{b_3} \tau_{b_4})$ is an open book decomposition for virtually overtwisted lens space $L(4,1)$. Since the curve $c$ sits on the page of this open book decomposition, it represents a Legendrian knot, $L$ in the virtually overtwisted $L(4,1)$. Now the manifold supported by $(\mathbf{D}_4, \tau_a \tau_b \tau_c \tau_{b_1} \tau_{b_3} \tau_{b_4})$ is obtained by a Legendrian surgery along the Legendrian knot $L$, in virtually overtwisted lens space $L(4,1)$. This proves that there is a unique diffeomorphism type for the Stein filling of our $3$ manifold. Now knowing that for any knot there are only finitely many Legendrian representative upto contactomorphism in any tight contact manifold implies that there are only finitely many Stein fillings of the given $3$ manifold.

\begin{figure}[htb]
\begin{center}

\begin{overpic}[tics=10]
{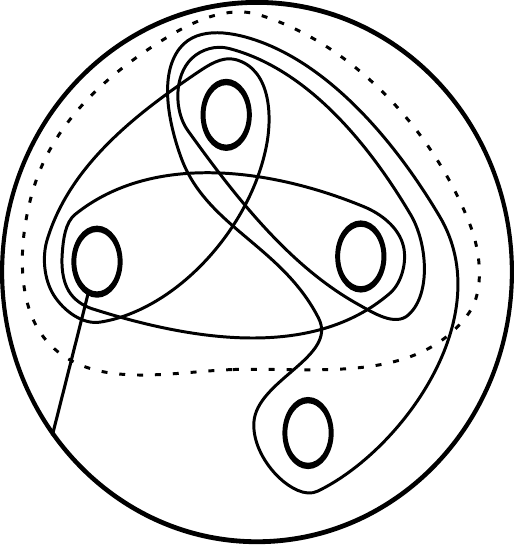}

\put(23, 78){$b_1$}
\put(62, 122){$b_2$}
\put(100, 79){$b_3$}
\put(85, 30){$b_4$}
\put(20, 110){$a$}
\put(92, 115){$b$}
\put(113, 25){$c$}
\put(60, 62){$d$}
\put(50, 40){$\gamma$}
\put(20, 35){$\alpha$}

\end{overpic}	
\caption{Curves $a,b,c,d,\gamma$ are used in getting infinite factorization of the word $\tau_a \tau_b \tau_c \tau_{b_1} \tau_{b_3} \tau_{b_4}$.}
\label{curious_example_fig}
\end{center}
\end{figure}
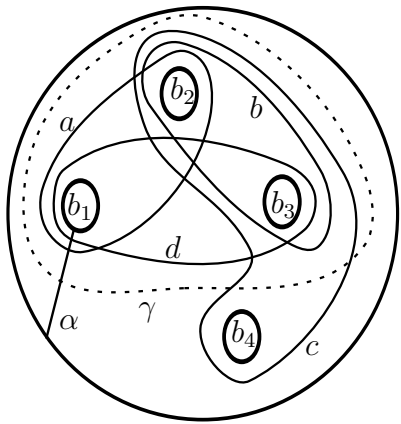

\end{proof}

\bibliography{myrefs}

\printindex

\end{document}